\documentclass[a4paper,12pt]{article}
\usepackage{amsmath}
\usepackage{amsthm}
\usepackage{amssymb}
\usepackage{enumerate}
\usepackage{url}

\title{On some identities in law involving exponential functionals of Brownian motion and Cauchy variable}

\author{Yuu Hariya\thanks{Mathematical Institute, 
Tohoku University, Aoba-ku, Sendai 980-8578, Japan. }}

\date{\empty}
\numberwithin{equation}{section}

\theoremstyle{plain}

\newtheorem{thm}{Theorem}[section]

\newtheorem{prop}{Proposition}[section]

\newtheorem{lem}{Lemma}[section]

\theoremstyle{definition}

\theoremstyle{remark}
\newtheorem{rem}{Remark}[section]
\newtheorem{exm}{Example}[section]

\begin{document}

\newcommand\ND{\newcommand}
\newcommand\RD{\renewcommand}

\ND\N{\mathbb{N}}
\ND\R{\mathbb{R}}
\ND\Q{\mathbb{Q}}
\ND\C{\mathbb{C}}

\ND\F{\mathcal{F}}

\ND\kp{\kappa}

\ND\ind{\boldsymbol{1}}

\ND\al{\alpha }
\ND\la{\lambda }
\ND\ve{\varepsilon}
\ND\Om{\Omega}

\ND\ga{\gamma }

\ND\lref[1]{Lemma~\ref{#1}}
\ND\tref[1]{Theorem~\ref{#1}}
\ND\pref[1]{Proposition~\ref{#1}}
\ND\sref[1]{Section~\ref{#1}}
\ND\ssref[1]{Subsection~\ref{#1}}
\ND\aref[1]{Appendix~\ref{#1}}
\ND\rref[1]{Remark~\ref{#1}} 
\ND\cref[1]{Corollary~\ref{#1}}
\ND\eref[1]{Example~\ref{#1}}
\ND\fref[1]{Fig.\ {#1} }
\ND\lsref[1]{Lemmas~\ref{#1}}
\ND\tsref[1]{Theorems~\ref{#1}}
\ND\dref[1]{Definition~\ref{#1}}
\ND\psref[1]{Propositions~\ref{#1}}
\ND\rsref[1]{Remarks~\ref{#1}}
\ND\sssref[1]{Subsections~\ref{#1}}

\ND\pr{\mathbb{P}}
\ND\ex{\mathbb{E}}
\ND\br{W}

\ND\eb[1]{e^{B_{#1}}}
\ND\ebm[1]{e^{-B_{#1}}}
\ND\hbe{\Hat{\beta}}
\ND\hB{\Hat{B}}
\ND\argsh{\mathrm{Argsh}\,}
\ND\zmu{z_{\mu}}
\ND\GIG[3]{\mathrm{GIG}(#1;#2,#3)}
\ND\gig[3]{I^{(#1)}_{#2,#3}}
\ND\argch{\mathrm{Argch}\,}

\ND\vp{\varphi}
\ND\eqd{\stackrel{(d)}{=}}
\ND\db[1]{B^{(#1)}}
\ND\da[1]{A^{(#1)}}

\ND\Ga{\Gamma}
\ND\calE{\mathcal{E}}
\ND\calD{\mathcal{D}}

\def\thefootnote{{}}

\maketitle 
\begin{abstract}
Let $B=\{ B_{t}\} _{t\ge 0}$ be a one-dimensional standard Brownian motion,  
to which we associate the exponential additive functional 
$A_{t}=\int _{0}^{t}e^{2B_{s}}ds,\,t\ge 0$. 
Starting from a simple observation of generalized inverse Gaussian 
distributions with particular sets of parameters, we show, with the help of 
a result by Matsumoto--Yor (2000), that for every $x\in \R $ and for 
every finite stopping time $\tau $ of the process 
$\{ \ebm{t}A_{t}\} _{t\ge 0}$, 
there holds the identity in law 
\begin{align*}
 &\left( 
 \eb{\tau}\!\sinh x+\beta (A_{\tau }), \, 
 C\eb{\tau}\!\cosh x+\hbe (A_{\tau }), \, 
 \ebm{\tau }\!A_{\tau }
 \right) \\
 &\stackrel{(d)}{=}
 \left( 
 \sinh (x+B_{\tau }), \, C\cosh (x+B_{\tau }), \, 
 \ebm{\tau }\!A_{\tau }
 \right) , 
\end{align*}
which extends an identity due to Bougerol (1983) in several aspects. 
Here $\beta =\{ \beta (t)\} _{t\ge 0}$ and $\hbe =\{ \hbe (t)\} _{t\ge 0}$ 
are one-dimensional standard Brownian motions, $C$ is a standard Cauchy 
variable, and $B$, $\beta $, $\hbe $ and $C$ are independent. 
Using an argument relevant to derivation of 
the above identity, we also present some invariance formulae 
for Cauchy variable involving an independent Rademacher 
variable. 
\footnote{E-mail: hariya@m.tohoku.ac.jp}
\end{abstract}

\section{Introduction}\label{;intro}

Let $B=\{ B_{t}\} _{t\ge 0}$ be a one-dimensional Brownian motion 
starting from $0$ and set 
\begin{align*}
 A_{t}:=\int _{0}^{t}e^{2B_{s}}ds, \quad t\ge 0. 
\end{align*}
This additive functional appears as the quadratic variation process 
of a geometric Brownian motion $\eb{t},\,t\ge 0$, and these exponential 
functionals of Brownian motion have importance in 
a number of fields such as option pricing in mathematical finance 
(see, e.g., \cite{gem}), diffusion processes in random environments 
(\cite{kt} and \cite{cmy} among others), stochastic analysis of 
Laplacians on hyperbolic spaces (see \cite[Sections~7.4 and 7.5]{mt} 
and references therein) and so on. There have been extensive studies 
made on these functionals (see the monograph \cite{yorm} by Yor 
and detailed surveys \cite{mySI, mySII} by Matsumoto and Yor) and 
various kinds of equalities and identities that give us deep 
understanding of their laws have been found, among which Bougerol's 
celebrated identity (\cite{bou}) states that for every fixed $t>0$, 
\begin{align}
 \beta (A_{t})&\eqd \sinh B_{t}, \label{;boug1}
\intertext{or more generally, for every fixed $t>0$ and $x\in \R $,}
 \eb{t}\sinh x+\beta (A_{t})&\eqd \sinh (x+B_{t}), \label{;boug2}
\end{align}
where $\beta =\{ \beta (t)\} _{t\ge 0}$ is a Brownian motion independent 
of $B$; unless otherwise stated, 
any Brownian motion that appears in this paper is one-dimensional and 
standard, namely starting from $0$. The former identity \eqref{;boug1} is 
particularly useful in deriving an explicit expression for the Mellin 
transform of the law of $A_{t}$. 

To see that the latter holds, one may follow the inventive reasoning 
due to Alili and Dufresne for the case $x=0$ provided in 
\cite[Appendix]{cmy}. For a fixed $x\in \R $, 
we consider the process $Y^{x}=\{ Y^{x}_{t}\} _{t\ge 0}$ given by 
\begin{align}\label{;defY}
 Y^{x}_{t}=\ebm{t}\sinh x +\ebm{t}\!\int _{0}^{t}e^{B_{s}}\,d\br _{s},\quad 
 t\ge 0, 
\end{align}
where $\br =\{ \br _{t}\} _{t\ge 0}$ is a Brownian motion independent of 
$B$. By defining a Brownian motion $\beta ^{x}=\{ \beta ^{x}_{t}\} _{t\ge 0}$ 
in such a way that 
\begin{align}\label{;defbex}
 \beta ^{x}_{t}=\int _{0}^{t}\frac{-Y^{x}_{s}\,dB_{s}+d\br _{s}}
 {\sqrt{1+(Y^{x}_{s})^{2}}}, 
\end{align}
It\^o's formula entails that $Y^{x}$ satisfies the following 
stochastic differential equation (SDE): 
\begin{align*}
 dY^{x}_{t}=\sqrt{1+(Y^{x}_{t})^{2}}\,d\beta ^{x}_{t}+\frac{1}{2}Y^{x}_{t}\,dt,
 \quad Y^{x}_{0}=\sinh x, 
\end{align*}
which is uniquely solved as 
\begin{align}\label{;expliy}
 Y^{x}_{t}=\sinh \left( x+\beta ^{x}_{t}\right) ,\quad t\ge 0. 
\end{align}
On the other hand, due to independence of $B$ and $\br $, we may 
express $Y^{x}$ as 
\begin{align}\label{;anothery}
 Y^{x}_{t}=\ebm{t}\sinh x+\ebm{t}\beta (A_{t}),\quad t\ge 0, 
\end{align}
with $\beta =\{ \beta (t)\} _{t\ge 0}$ another Brownian motion independent of 
$B$. Therefore for every fixed $t>0$, 
\begin{align}
 Y^{x}_{t}&\eqd \ebm{t}\sinh x+\beta (e^{-2B_{t}}\!A_{t}) \notag \\
 &\eqd \eb{t}\sinh x+\beta (A_{t}), \label{;ideny}
\end{align}
where the first line is due to the scaling property of Brownian motion 
and the second follows from the identity in law 
\begin{align}\label{;2-dim}
 \left( \ebm{t},\,e^{-2B_{t}}\!A_{t}\right) 
 \eqd \left( \eb{t},\,A_{t}\right) 
\end{align}
thanks to the time reversal: 
$\{ B_{t}-B_{t-s}\} _{0\le s\le t}\eqd 
\{ B_{s}\} _{0\le s\le t}$. Comparing \eqref{;ideny} and \eqref{;expliy}  
leads to \eqref{;boug2}. For further details as well as recent progress in 
the research on Bougerol's identity such as extensions to other 
processes, we refer the reader to the survey \cite{vak} by Vakeroudis; 
for a matrix-valued extension of the identity, see \cite{tas}. 

Following the notation in a series of papers 
\cite{myPI,myPII,my2003,mySII} by Matsumoto--Yor, 
we denote by $Z=\{ Z_{t}\} _{t\ge 0}$ the process defined by 
\begin{align*}
 Z_{t}:=\ebm{t}A_{t}. 
\end{align*}
As studied in detail in \cite{myPI,myPII}, this process is 
a diffusion process in its own natural filtration, 
which will be recalled in \rref{;rpparti} below. 
In the sequel we also denote by $C$ a standard Cauchy 
variable whose probability density is 
$\{ \pi (1+x^{2})\} ^{-1},\,x\in \R $. 
Given a real-valued process $X=\{ X_{t}\} _{t\ge 0}$ 
and a point $a\in \R $, we denote by $\tau _{a}(X)$ the first 
hitting time of $X$ to $a$: 
\begin{align*}
 \tau _{a}(X):=\inf \{ t\ge 0;\,X_{t}=a\}  
\end{align*}
with convention that  $\tau _{a}(X)=\infty $ when 
$\{ \} =\emptyset $. One of the objectives of this paper is 
to show that Bougerol's identity 
may be extended in the following manner: in the statement below, 
three processes 
$\beta =\{ \beta (t)\} _{t\ge 0}$, $\hbe =\{ \hbe (t)\} _{t\ge 0}$ and 
$\hB =\{ \hB _{t}\} _{t\ge 0}$ denote Brownian motions. 

\begin{thm}\label{;tm1}
 Fix $x\in \R $. For any stopping time $\tau $ of the process $Z$ 
 such that $0<\tau <\infty $ a.s., we have 
 \begin{equation}\label{;eqtm1}
  \begin{split}
  &\left( 
  \eb{\tau}\!\sinh x+\beta (A_{\tau }), \, 
  C\eb{\tau}\!\cosh x+\hbe (A_{\tau }), \, 
  Z_{\tau }
  \right) \\
  &\eqd 
  \left( 
  \sinh (x+B_{\tau }), \, C\cosh (x+B_{\tau }), \, 
  Z_{\tau }
  \right) , 
  \end{split}
 \end{equation}
 or equivalently, 
 \begin{equation}\label{;eqtm1d}
 \begin{split}
  &\left( 
  \eb{\tau}\!\sinh x+\beta (A_{\tau }), \, 
  \tau _{\eb{\tau }\!\cosh x}(\hB )+A_{\tau }, \, 
  Z_{\tau }
  \right) \\
  &\eqd 
  \left( 
  \sinh (x+B_{\tau }), \, \tau _{\cosh (x+B_{\tau })}(\hB ), \, 
  Z_{\tau }
  \right) , 
  \end{split}
 \end{equation}
 where in \eqref{;eqtm1}, $B$, $\beta $, $\hbe $ and $C$ (resp.~$B$ 
 and $C$) are independent on the left-(resp. right-)hand side while 
 in \eqref{;eqtm1d}, $B$, $\beta $ and $\hB $ (resp.~$B$ and $\hB $) 
 are independent on the left-(resp.~right-)hand side. 
\end{thm}

\begin{rem}\label{;rstop}
\thetag{1} Once \tref{;tm1} is established, then its extension to 
the case where $\tau $ is only assumed to be finite a.s.\ is 
straightforward thanks to sample path continuity of 
$B$, $\beta $ and $\hbe $; indeed, if $\tau $ is such that 
$\tau <\infty $ a.s., then by \tref{;tm1}, the identity \eqref{;eqtm1} 
holds with $\tau $ replaced by $\max \{ \tau, \delta \} $ 
for any $\delta >0$ and the above-mentioned extension follows
by letting $\delta \to 0$. The same remark 
stands as to the identity \eqref{;eqtm1d} by noting the fact that 
$\tau _{a}(B)\eqd a^{2}\tau _{1}(B)$ for any $a\in \R $ because of the 
scaling property of Brownian motion (see also \eqref{;idenst} below). 

\noindent 
\thetag{2} It should be noted that the natural filtration of 
the process $Z$ is strictly contained in that of the original 
Brownian motion $B$ (see \cite[Theorem~1.6]{myPI}); therefore it is 
true that $\tau $ in the statement of the theorem is a stopping time 
of $B$ but the converse is not true. 

\noindent 
\thetag{3} It would be interesting to note that taking $x=0$ in 
\eqref{;eqtm1}, we have in particular 
\begin{align}\label{;eqrstop1}
 \frac{\beta (A_{t})}{Z_{t}}\eqd \frac{\sinh B_{t}}{Z_{t}}
\end{align}
for every fixed $t>0$, from which it follows that 
\begin{align}\label{;eqrstop2}
 \beta \left( \frac{1}{A_{t}}\right) 
 \eqd \frac{1}{2}\left( 
 \frac{e^{2B_{t}}}{A_{t}}-\frac{1}{A_{t}}
 \right) . 
\end{align}
Indeed, thanks to the scaling property of Brownian motion and 
\eqref{;2-dim}, the left-hand side of \eqref{;eqrstop1} 
is identical in law with 
\begin{align*}
 \beta \left( \frac{e^{2B_{t}}}{A_{t}^{2}}\cdot A_{t}\right) 
 \eqd \beta \left( \frac{1}{A_{t}}\right) . 
\end{align*}
It is informative to mention that in view of \eqref{;2-dim}, the right-hand 
side of \eqref{;eqrstop2} does give a symmetric random variable: 
\begin{align*}
 -\left( 
 \frac{e^{2B_{t}}}{A_{t}}-\frac{1}{A_{t}}
 \right) \eqd 
 \frac{e^{2B_{t}}}{A_{t}}-\frac{1}{A_{t}}. 
\end{align*}
Similarly to \eqref{;eqrstop2}, it also holds that for every fixed $t>0$, 
\begin{align}\label{;eqrstop3}
 \beta (\eb{t})\eqd 
 \frac{1}{2}\left( 
 \sqrt{\frac{e^{3B_{t}}}{A_{t}}}
 -\sqrt{\frac{\ebm{t}}{A_{t}}}
 \right) 
\end{align}
if we consider $\bigl( 1/\sqrt{Z_{t}}\bigr) \sinh B_{t}$. 
\end{rem}

We also prove another extension of Bougerol's identity: given a 
real-valued process $X=\{ X_{t}\} _{t\ge 0}$ and $\mu \in \R $, 
we denote $X^{(\mu )}=\{ X^{(\mu )}_{t}:=X_{t}+\mu t\} _{t\ge 0}$. 

\begin{thm}\label{;tm2}
 Fix $x\in \R $. For any stopping time $\tau $ of the process $Z$ such 
 that $0<\tau <\infty $ a.s., it holds that 
 \begin{equation}\label{;eqtm2}
  \begin{split}
  &\left( \eb{\tau}\!\sinh x+\beta (A_{\tau }),\,\eb{\tau },\,Z_{\tau }\right) \\
  &\eqd 
  \left( 
  \sinh (x+B_{\tau }),\,
  \tau _{\cosh (x+B_{\tau })}(\hB ^{(\cosh x/Z_{\tau })})/Z_{\tau },
  \,Z_{\tau }
  \right) , 
  \end{split}
 \end{equation}
 where both $\beta =\{ \beta (t)\} _{t\ge 0}$ and 
 $\hB =\{ \hB _{t}\} _{t\ge 0}$ are Brownian motions independent of 
 $B$. 
\end{thm}

In Alili--Gruet \cite{ag}, it is shown that for every fixed $t>0$, 
\begin{align}\label{;agi}
 \left( \beta (A_{t}),\,B_{t}\right) 
 \eqd 
 \left( 
 (2Y-1)\phi \bigl( 
 B_{t},\sqrt{R_{t}^{2}+B_{t}^{2}}\,
 \bigr) ,\,B_{t}
 \right) , 
\end{align}
where on the right-hand side, $Y$ is an arcsine 
variable whose probability density function is 
$\bigl\{ \pi \sqrt{x(1-x)}\bigr\} ^{-1},\,0<x<1$, the function 
$\phi $ is defined by 
\begin{align*}
 \phi (x,z)=\sqrt{2e^{x}(\cosh z-\cosh x)}
\end{align*}
for two reals $x$ and $z$ fulfilling $z\ge |x|$, 
the process 
$R=\{ R_{t}\} _{t\ge 0}$ is a two-dimensional Bessel process 
starting from $0$, and three elements $B$, $Y$ and $R$ are independent; 
for the identity \eqref{;agi} and its proof, we also 
refer to \cite[Lemma~1]{myPJ}. 
In the case $x=0$, the identity \eqref{;eqtm2} in 
the above theorem complements \eqref{;agi} in the sense that 
it keeps the expressions of first coordinates the same as in 
Bougerol's  original identity \eqref{;boug1}. 

Apart from the proof of \tref{;tm2} given in the next section, 
in order to convince the reader of the validity of \eqref{;eqtm2}, 
we recall the well-known fact that for every 
$a,\mu \in \R $ with $a\mu \ge 0$ and for every $\la \in \R $, 
\begin{align}\label{;lt}
 \ex \!\left[ 
 \exp \left\{ 
 -\frac{\la ^{2}}{2}\tau _{a}(\db{\mu })
 \right\} 
 \right] 
 =\exp \left\{ 
 -\bigl( 
 \sqrt{\mu ^{2}+\la ^{2}}-\mu 
 \bigr) |a|
 \right\} 
\end{align}
(see, e.g., \cite[Exercise~3.5.10]{ks}, \cite[p.~301, Formula~2.0.1]{bs}), 
from which we easily deduce, by differentiating both sides with respect to 
$\la $, that when $\mu \neq 0$, 
\begin{align*}
 \ex [\tau _{a}(\db{\mu })]=\frac{a}{\mu }. 
\end{align*}
Thanks to this formula, we see that for every $x\in \R $ and $t>0$, 
\begin{align*}
 \ex \!\left[ 
 \tau _{\cosh (x+B_{t})}(\hB ^{(\cosh x/Z_{t})})/Z_{t}
 \right] &=\frac{\ex [\cosh (x+B_{t})]}{\cosh x}\\
 &=\frac{1}{2\cosh x}\left\{ 
 e^{x}\ex [\eb{t}]+e^{-x}\ex [\ebm{t}]
 \right\} , 
\end{align*}
which, regardless of $x$, agrees with $\ex [\eb{t}]$ by symmetry of 
Brownian motion. 

Proof of our \tsref{;tm1} and \ref{;tm2} hinges upon the following 
simple observation: 

\begin{prop}\label{;phinge}
 Given $\mu >0$, let $f:\R \to \R $ be a measurable function 
 such that 
 \begin{align*}
  \int _{\R }dy\,e^{-\mu \cosh y}\left| f(\sinh y)\right| <\infty . 
 \end{align*}
 Then it holds that 
 \begin{align}\label{;eqphinge}
  \int _{\R }dy\,e^{-\mu \cosh y}f(\sinh y)
  =\int _{0}^{\infty }\frac{dv}{v}\,\exp \left( 
  -\frac{1}{2v}\right) 
  \exp \left( -\frac{\mu ^{2}}{2}v\right) 
  \ex \!\left[ 
  f(\beta (v))
  \right] . 
 \end{align}
 Here $\beta =\{ \beta (t)\} _{t\ge 0}$ is a Brownian motion. 
\end{prop}

\begin{rem}\label{;rrel}
 Although not stated in an explicit manner in the existing 
 literature, this is not the first time for the relation 
 \eqref{;eqphinge} to be noticed; see \rref{;rlaws} for more 
 information. 
\end{rem}

We give a proof of the above proposition in the next section. 
The proposition also enables us to obtain 
some invariance formulae for Cauchy variable, which are of interest 
in their own right: 

\begin{thm}\label{;tm3}
 Let $\ve $ be a Rademacher (or 
 symmetric Bernoulli) variable taking values $\pm 1$ with 
 probability $1/2$ independently of the standard Cauchy variable $C$. 
 It then holds that: 
 \begin{enumerate}[(i)]{}
  \item for every $a\in \R $ with $|a|\ge 1$ and for every 
  $\theta \in [-1,1]$, 
  \begin{align*}
   aC+\theta \sqrt{1+a^{2}C^{2}}\,\ve 
   \eqd 
   \bigl( 
   a+\theta \sqrt{a^{2}-1}\,\ve 
   \bigr) C;
  \end{align*}
 
  \item for every $a,\theta \in [-1,1]$, 
  \begin{align*}
   aC+\theta \sqrt{1+a^{2}C^{2}}\,\ve 
   \eqd 
   aC+\theta \sqrt{1-a^{2}}\,\ve . 
  \end{align*}
 \end{enumerate}
\end{thm}

As both sides of  the claimed identities contain 
$aC$ in common, we refer to them as invariance 
formulae. In \ssref{;partial}, we show that by applying 
these invariance formulae, identities in \tref{;tm1} are 
recovered partly. As for invariance of Cauchy variable, the 
following one would also be of interest: 
\begin{align}\label{;invcauchy}
 \frac{
 C\eb{\tau}\!\cosh x+\hbe (A_{\tau })
 }
 {
 \sqrt{
 1+\left( 
 \eb{\tau}\!\sinh x+\beta (A_{\tau })
 \right) ^{2}
 }
 }
 \eqd C, 
\end{align}
which is an immediate consequence of the identity 
\eqref{;eqtm1} in \tref{;tm1}. Here on the left-hand side, four 
elements $B$, $C$, $\beta $ and $\hbe $ are independent, and $x$ 
is any real as well as $\tau $ is any finite 
stopping time of $Z$ (recall \rref{;rstop}\,\thetag{1}); 
notice that the random variable on the left-hand side 
of \eqref{;invcauchy} is independent of $Z_{\tau }$ as 
\eqref{;eqtm1} indicates. See also \rref{;rpkey}\,\thetag{3} below. 

We give an outline of the paper. 
In \sref{;prfs}, 
we give proofs of \pref{;phinge} and \tsref{;tm1}--\ref{;tm3}: 
we prove \pref{;phinge} and recall the definition of 
generalized inverse Gaussian (GIG for short) distributions 
in \ssref{;pphinge}; in \ssref{;ptms}, we prove \tsref{;tm1} 
and \ref{;tm2} preparing \psref{;pparti} and \ref{;pkey}, 
the former of which is due to Matsumoto--Yor \cite{myPI} and 
the latter of which concerns some properties of 
GIG laws relevant to the two theorems; 
in \ssref{;prftm3}, we give a proof of \tref{;tm3} by utilizing 
a particular case of \lref{;lkey}, whose assertion follows from 
\pref{;phinge} and is also used in the proof of \pref{;pkey}. 
In \sref{;related}, we provide some 
results related to, as well as deduced from, our discussions 
developed in \sref{;prfs}. One of them is the derivation, 
in a self-contained way, of an identity in law due to Dufresne 
\cite{duf} that is another profound identity in the study of 
exponential functionals of Brownian motion. In \sref{;cr}, 
we give concluding remarks in relation to the SDE approach 
to the identity \eqref{;boug2} introduced at the beginning of 
the present section. Finally in the 
appendix, we explore several facts relevant to the (unnormalized) 
density function of the so-called Hartman--Watson law, which 
appears in an explicit representation for the joint 
law of $\eb{t}$ and $A_{t}$ given $t>0$, due to Yor \cite{yor92}. 

Throughout the paper, all random variables as well as all 
stochastic processes are assumed to be defined on a common 
probability space whose probability measure is denoted by $\pr $. 
Expectation relative to $\pr $ is denoted by $\ex $. We also 
suppose that the probability space we work in is nice 
enough to support regular conditional probabilities 
which we denote by the symbol $\pr (\ |\ )$. We say that a random 
variable $X$ is \emph{symmetric} if $-X$ has the same law as $X$: 
$-X\eqd X$. Other notation and terminology will be introduced 
as needed. 

\section{Proofs}\label{;prfs}

We devote this section to proofs of the results introduced in the 
previous section. 

\subsection{Proof of \pref{;phinge}}\label{;pphinge}
In this subsection, we prove \pref{;phinge} and restate it 
in terms of random variables, to which we relate the notion of 
generalized inverse Gaussian distributions. 
We recall that for every 
real $a\neq 0$, the law of $\tau _{a}(B)$ is given by 
\begin{align}\label{;taulaw}
 \pr \left( \tau _{a}(B)\in dv\right) 
 =\frac{|a|}{\sqrt{2\pi v^{3}}}\exp \left( 
 -\frac{a^{2}}{2v}
 \right) dv,\quad v>0, 
\end{align}
which may be seen from \eqref{;lt} as well as from the 
fact that by reflection principle, 
\begin{align}\label{;idenst}
 \tau _{a}(B)\eqd \frac{a^{2}}{B_{1}^{2}}. 
\end{align} 
In the sequel we denote by 
\begin{align*}
 \argsh x\equiv \log \bigl( x+\sqrt{1+x^{2}}\bigr) ,\quad x\in \R , 
\end{align*}
the inverse function of the hyperbolic sine function. 

\begin{proof}[Proof of \pref{;phinge}]
 By \eqref{;lt} and \eqref{;taulaw}, we may write 
 \begin{align*}
  e^{-\mu \cosh y}
  &=\ex \Bigl[ e^{-\frac{\mu ^{2}}{2}\tau _{\cosh y}(B)}\Bigr] \\
  &=\int _{0}^{\infty }dv\,\frac{\cosh y}{\sqrt{2\pi v^{3}}}\exp \left( 
  -\frac{\cosh ^{2}y}{2v}
  \right) \exp \left( -\frac{\mu ^{2}}{2}v\right) 
 \end{align*}
 for every $\mu >0$ and $y\in \R $. Plugging the last expression 
 into the left-hand side of \eqref{;eqphinge} and using Fubini's 
 theorem, we have the equality 
 \begin{align*}
  &\int _{\R }dy\,e^{-\mu \cosh y}f(\sinh y)\\
  &=\int _{0}^{\infty }\frac{dv}{v}\,\exp \left( -\frac{1}{2v}\right) 
  \exp \left( -\frac{\mu ^{2}}{2}v\right) 
  \int _{\R }dy\,\frac{\cosh y}{\sqrt{2\pi v}}
  \exp \left( -\frac{\sinh ^{2}y}{2v}\right) f(\sinh y). 
 \end{align*}
 Changing the variables with $y=\argsh x,\,x\in \R $, in the 
 integral with respect to $y$, leads to the desired expression. 
\end{proof}

It is possible to rephrase \eqref{;eqphinge} as an identity 
in law. For every $\mu >0$, we consider 
a random variable $\zmu $ whose law is given by 
\begin{align}\label{;zlaw}
 \pr \left( \zmu \in dx\right) 
 =\frac{1}{2K_{0}(\mu )}e^{-\mu \cosh x}\,dx,\quad x\in \R . 
\end{align}
Here and in what follows, for every $\nu \in \R $, we denote by 
$K_{\nu }$ the modified Bessel function of the third kind 
(Macdonald function) 
of order $\nu $, one of whose integral representations is given by 
\begin{align}\label{;intk1}
 K_{\nu }(z)
 =\frac{1}{2}\int _{\R }dx\,e^{-z\cosh x-\nu x}
 =\frac{1}{2}\int _{\R }dx\,e^{-z\cosh x}\cosh (\nu x),\quad 
 z>0 
\end{align}
(cf. \cite[Equation~(5.10.23)]{leb}). The statement of \pref{;phinge} 
is then rephrased as 
\begin{align}\label{;proto1}
 \sinh \zmu \eqd \beta \left( \frac{e^{\zmu }}{\mu }\right) 
\end{align}
for every $\mu >0$, where on the right-hand side, 
$\beta =\{ \beta (t)\} _{t\ge 0}$ is a 
Brownian motion independent of $\zmu $; indeed, we divide by 
$2K_{0}(\mu )$ both sides of \eqref{;eqphinge} and change the 
variables with $v=e^{y}/\mu ,\,y\in \R $, on the right-hand side 
to see that 
\begin{equation}\label{;proto1d}
 \begin{split}
 \ex \!\left[ 
 f(\sinh \zmu )
 \right] 
 &=\frac{1}{2K_{0}(\mu )}\int _{\R }dy\,e^{-\mu \cosh y}
 \ex \!\left[ 
 f(\beta (e^{y}/\mu ))
 \right] \\
 &=\ex \!\left[ 
 f(\beta (e^{\zmu }/\mu ))
 \right] 
 \end{split}
\end{equation}
for any bounded measurable function $f$ on $\R $. As will be seen in the 
next subsection, if we let $\mu $ vary according to the law of 
$1/Z_{t}\equiv \eb{t}/A_{t}$ for a fixed $t>0$, or in other words, 
if we integrate both sides of \eqref{;proto1d} over $\mu >0$ 
with respect to $\pr (\eb{t}/A_{t}\in d\mu )$, then 
Bougerol's identity \eqref{;boug1} is recovered. 

We may relate the random variable $\zmu $ to a generalized 
inverse Gaussian (GIG) law: recall from 
\cite[Section~9]{myPII} that given three parameters 
$\nu \in \R $ and $a,b>0$, a random variable 
$I\equiv \gig{\nu }{a}{b}$ is said to follow the 
$\GIG{\nu }{a}{b}$ distribution if 
\begin{align}\label{;gig}
 P(I\in dv)
 =\left( \frac{b}{a}\right) ^{\nu }
 \frac{v^{\nu -1}}{2K_{\nu }(ab)}
 \exp \left\{ 
 -\frac{1}{2}\left( 
 \frac{a^{2}}{v}+b^{2}v
 \right) 
 \right\} dv,\quad v>0. 
\end{align}
To verify that the right-hand side does give a probability 
distribution, it is meaningful to recall another integral 
representation of $K_{\nu }$, which one obtains 
from \eqref{;intk1} by changing the variables 
with $e^{-x}=2v/z,\,v>0$: 
\begin{align}\label{;intk2}
 K_{\nu }(z)=\frac{1}{2}\left( 
 \frac{z}{2}
 \right) ^{-\nu }\!\int _{0}^{\infty }dv\,
 v^{\nu -1}\exp \left( 
 -v-\frac{z^{2}}{4v}
 \right) ,\quad z>0
\end{align}
(see e.g., \cite[Equation~(5.10.25)]{leb}); 
a probabilistic expression for \eqref{;intk2} in terms of 
gamma variable, will be given in \eqref{;intk2d}.  
In a similar way to \eqref{;proto1d}, we 
see that $e^{\zmu }/\mu $ is $\GIG{0}{1}{\mu }$-distributed: 
\begin{align}\label{;idengig}
 \frac{e^{\zmu }}{\mu }\eqd \gig{0}{1}{\mu }, \quad 
 \text{or equivalently,}\ 
 e^{\zmu }\eqd \gig{0}{\sqrt{\mu }}{\sqrt{\mu }}. 
\end{align}
Indeed, for any bounded measurable function $f$ on 
$(0,\infty )$, we have 
\begin{align*}
 \ex \!\left[ 
 f(e^{z_{\mu }}/\mu )
 \right] 
 &=\frac{1}{2K_{0}(\mu )}\int _{\R }dx\,e^{-\mu \cosh x}f(e^{x}/\mu )\\
 &=\frac{1}{2K_{0}(\mu )}\int _{0}^{\infty }
 \frac{dv}{v}\,\exp \left( -\frac{1}{2v}\right) 
 \exp \left( -\frac{\mu ^{2}}{2}v\right) f(v)\\
 &=\ex \bigl[ 
 f(\gig{0}{1}{\mu })
 \bigr] 
\end{align*}
by \eqref{;gig}, which shows the former identity in \eqref{;idengig}. 
Here for the second line, we changed the variables 
with $e^{x}/\mu =v$. The latter follows from the elementary property that 
$c^{2}\gig{\nu }{a}{b}\eqd \gig{\nu }{ac}{a/c}$ for every 
$\nu \in \R $ and $a,b,c>0$. The above computation also reveals the 
following representation of $K_{0}$, which is immediate from 
\eqref{;intk2} as well, and will be recalled repeatedly: 
\begin{align}\label{;k0}
 K_{0}(z)=\frac{1}{2}\int _{0}^{\infty }\frac{dv}{v}\,
 \exp \left( -\frac{1}{2v}\right) 
 \exp \left( -\frac{z^{2}}{2}v\right) ,\quad z>0. 
\end{align}
In the case $\nu =0$, the GIG distribution is also referred to 
as Halphen's harmonic law or the hyperbola distribution; 
see \cite{jor} and \cite{ses} (be aware that the above 
parametrization for GIG laws is slightly different from the one 
used in these references). 
We also remark that when $a,\mu >0$, the stopping time 
$\tau _{a}(\db{\mu })$ as appeared in \eqref{;lt} follows the 
inverse Gaussian distribution, namely it is identical 
in law with $\gig{-1/2}{a}{\mu }$. 

For later use in \rref{;rpkey}, it is convenient to observe here that 
the following convergence in law takes place: 

\begin{prop}\label{;pconvlaw}
 $\gig{0}{\sqrt{\mu }}{\sqrt{\mu }}$ converges in law 
 to $1$ as $\mu \to \infty $. The same convergence holds true for 
 $e^{\zmu }$ by virtue of \eqref{;idengig}. 
\end{prop}

\begin{proof}
 Let $f:(0,\infty )\to \R $ be bounded and continuous and note the 
 equalities 
 \begin{align*}
  &\ex \bigl[ 
  f(\gig{0}{\sqrt{\mu }}{\sqrt{\mu }})
  \bigr] \times \sqrt{\frac{2\mu }{\pi }}
  e^{\mu }K_{0}(\mu )
  \\
  &=\sqrt{\frac{\mu }{2\pi }}\int _{0}^{\infty }\frac{dv}{v}\,
  \exp \left\{ 
  -\frac{\mu }{2}\left( \sqrt{v}-\frac{1}{\sqrt{v}}\right) ^{2}
  \right\} f(v)\\
  &=\int _{\R }\frac{dx}{\sqrt{2\pi }}\,
  \exp \left( -\frac{x^{2}}{2}\right) 
  \frac{1}{\sqrt{1+x^{2}/(4\mu )}}
  f\left( \left( 
  \frac{x}{2\sqrt{\mu }}+\sqrt{1+\frac{x^{2}}{4\mu }}\,
  \right) ^{2}\right) , 
 \end{align*}
 where we changed the variables with 
 $\sqrt{v}-1/\sqrt{v}=x/\sqrt{\mu },\,x\in \R $, for the 
 last line. We let $\mu \to \infty $: it is known that for any $\nu \in \R $, 
 \begin{align}\label{;klim}
  \sqrt{\frac{2\mu }{\pi }}
  e^{\mu }K_{\nu }(\mu )\to 1 
 \end{align}
 (see \cite[Equation~(5.16.5)]{leb}); on the other hand, thanks to 
 boundedness and continuity of $f$, the bounded convergence theorem 
 entails that the last expression of the above equalities converges to 
 \begin{align*}
  \int _{\R }\frac{dx}{\sqrt{2\pi }}\exp \left( -\frac{x^{2}}{2}\right) 
  \times f(1)=f(1). 
 \end{align*}
 Therefore we have 
 \begin{align*}
  \lim _{\mu \to \infty }\ex \bigl[ 
  f(\gig{0}{\sqrt{\mu }}{\sqrt{\mu }})
  \bigr] =f(1)
 \end{align*}
 for any bounded and continuous $f$ on $(0,\infty )$, which 
 shows the proposition. 
\end{proof}

The above proposition may also be deduced from 
Equation~\thetag{3.10} or \thetag{3.11} in \cite{jor}; 
see \rref{;rpkey}\,\thetag{2} below as to those two 
equations. 

\subsection{Proof of \tsref{;tm1} and \ref{;tm2}}\label{;ptms}

In this subsection we prove \tsref{;tm1} and \ref{;tm2}. We begin with 
a proposition, the assertion of which is a particular case of that of 
\cite[Proposition~1.7]{myPI} as we are dealing with the Brownian motion 
without drift. 

\begin{prop}\label{;pparti}
 Let $\tau $ be a stopping time of the process $Z$ such that 
 $0<\tau <\infty $ a.s.\ and $\mu $ a positive real. Then 
 the conditional distribution of $\eb{\tau }/\mu $ given 
 $1/Z_{\tau }=\mu $ coincides with the $\GIG{0}{1}{\mu }$ 
 distribution; in other words, conditionally on $1/Z_{\tau }=\mu $, 
 it holds that  
 \begin{align}\label{;idenbz}
  B_{\tau }\eqd \zmu 
 \end{align}
 in view of \eqref{;idengig}. 
\end{prop}

\begin{rem}\label{;rsym}
 Since $z_{\mu }$ is symmetric for each $\mu >0$ 
 as seen from \eqref{;zlaw}, 
 the relation \eqref{;idenbz} indicates that $B_{\tau }$ is 
 symmetric, which fact may be 
 regarded as a reflection of strict containedness of 
 the natural filtration of $Z$ in that of $B$ referred to 
 in \rref{;rstop}\,\thetag{2}.  
\end{rem}

To keep the paper self-contained as much as possible, 
we provide in \rref{;rpparti} below reasoning to deduce 
\pref{;pparti} from the diffusion property of $Z$ 
discussed in detail in \cite{myPI,myPII}. 
Thanks to the proposition, \tsref{;tm1} and \ref{;tm2} 
are immediate consequences of the following 
extensions of the identity \eqref{;proto1}: as in the 
statements of those two theorems, we let 
$\beta =\{ \beta (t)\} _{t\ge 0}$, 
$\hbe =\{ \hbe (t)\} _{t\ge 0}$ and 
$\hB =\{ \hB _{t}\} _{t\ge 0}$ denote Brownian motions. 

\begin{prop}\label{;pkey}
 For every fixed $\mu >0$ and $x\in \R $, we have the following 
 identities in law: 
 \begin{enumerate}[(i)]{}
  \item it holds that 
  \begin{equation}\label{;eqpkey1}
   \begin{split}
    &\left( 
    e^{\zmu }\sinh x+\beta (e^{\zmu }/\mu ),\,
    Ce^{\zmu }\cosh x+\hbe (e^{\zmu }/\mu )
    \right) \\
    &\eqd 
    \left( 
    \sinh (x+\zmu ),\,C\cosh (x+\zmu )
    \right) , 
   \end{split}
  \end{equation}
  or equivalently, 
  \begin{equation}\label{;eqpkey2}
   \begin{split}
    &\left( 
    e^{\zmu }\sinh x+\beta (e^{\zmu }/\mu ),\,
    \tau _{e^{\zmu }\cosh x}(\hB )+e^{\zmu }/\mu 
    \right) \\
    &\eqd 
    \left( 
    \sinh (x+\zmu ),\,\tau _{\cosh (x+\zmu )}(\hB )
    \right) , 
   \end{split}
  \end{equation}
  where in \eqref{;eqpkey1}, $\zmu $, $\beta $, $\hbe $ and $C$ 
  (resp. $\zmu $ and $C$) are independent on the left-(resp. right-)hand 
  side while in \eqref{;eqpkey2}, $\zmu $, $\beta $ and $\hB $ 
  (resp. $\zmu $ and $\hB $) are independent on the 
  left-(resp. right-)hand side; 
  
  \item it holds that 
  \begin{equation}\label{;eqpkey3}
   \begin{split}
   &\left( 
   e^{\zmu }\sinh x+\beta (e^{\zmu }/\mu),\,e^{\zmu }
   \right) \\
   &\eqd 
   \left( 
   \sinh (x+\zmu ),\,
   \mu \tau _{\cosh (x+\zmu )}(\hB ^{(\mu \cosh x)})
   \right) , 
   \end{split}
  \end{equation}
  where $\beta $ and $\hB $ are independent of $\zmu $. 
 \end{enumerate}
\end{prop}

\begin{proof}[Proof of \tsref{;tm1} and \ref{;tm2}]
 Assertions of \tsref{;tm1} and \ref{;tm2} follow immediately 
 from \thetag{i} and \thetag{ii} of \pref{;pkey}, 
 respectively, since it holds that by \pref{;pparti},  
 \begin{align*}
  \int _{0}^{\infty }
  \pr \left( 1/Z_{\tau }\in d\mu \right) 
  \ex \!\left[ 
  f(\zmu ,e^{\zmu }/\mu ,1/\mu )
  \right] 
  =\ex \!\left[ 
  f(B_{\tau },A_{\tau },Z_{\tau })
  \right] 
 \end{align*}
 for any bounded measurable function 
 $f:\R \times (0,\infty )^{2}\to \R $. 
\end{proof}

\begin{rem}\label{;rlaws}
 If we denote by $N$ a one-dimensional standard Gaussian random 
 variable independent of $\gig{0}{1}{\mu }$, then by \eqref{;idengig}, 
 the first coordinate in the left-hand side of \eqref{;eqpkey1} 
 is identical in law with 
 \begin{align}\label{;mixture}
  \mu \sinh x\cdot \gig{0}{1}{\mu }+\sqrt{\gig{0}{1}{\mu }}N. 
 \end{align}
 If we replace $\gig{0}{1}{\mu }$ by a generic nonnegative 
 random variable $X$ and $\mu \sinh x$ by a generic constant, 
 then by adding an additional constant, the operation as in 
 \eqref{;mixture} of producing a new probability distribution 
 is often called the normal mean-variance mixture 
 with mixing law $X$ in the literature. For instance, in the final section 
 of \cite{bn77}, Barndorff-Nielsen uses $X=\gig{\nu }{a}{b}$ as 
 the mixing law to introduce briefly what is referred to as the generalized 
 hyperbolic distribution; see also \cite{jor} and \cite{ses}. 
 Although it is not stated in an explicit form as the identity between 
 first coordinates in \eqref{;eqpkey1}, a connection between the generalized 
 hyperbolic and hyperbola distributions is discussed later by Barndorff-Nielsen 
 \cite[Section~5]{bn78}. On the other hand, the identity between 
 second coordinates in \eqref{;eqpkey1}, as well as in \eqref{;eqpkey2}, 
 and that in \eqref{;eqpkey3} seem to be new to our knowledge; 
 in particular, the latter identity 
 reveals a connection between the hyperbola and inverse 
 Gaussian distributions. Moreover, the identity \eqref{;eqpkey3} 
 enables us to obtain a certain relationship among the laws of 
 $\sinh (x+\zmu )$ for different values of $x$; see \rref{;rpkey}\,\thetag{4}. 
 As for the notions of the hyperbola and inverse Gaussian distributions, 
 recall their description given just before 
 \pref{;pconvlaw}. 
\end{rem}

Before providing a proof of \pref{;pkey}, we explain 
how to deduce \pref{;pparti}, and provide related 
facts. 

\begin{rem}\label{;rpparti}
\thetag{1} First observe that 
$\lim \limits_{t\to \infty }A_{t}=\infty $ a.s.; indeed, by 
Lamperti's well-known relation (see, e.g., 
\cite[Chapter~XI, Exercise~\thetag{1.28}]{ry}), there exists 
a two-dimensional Bessel process $R=\{ R(t)\} _{t\ge 0}$ 
starting from 1, such that 
\begin{align*}
 \eb{t}=R(A_{t}),\quad t\ge 0, 
\end{align*}
from which absurdity of  
$\pr \left( \lim \limits_{t\to \infty }A_{t}<\infty \right) >0$ 
follows since this positivity should imply existence of 
$\lim \limits_{t\to \infty }B_{t}$ 
with positive probability due to sample path continuity of $R$. 
Therefore we have the relation 
\begin{align}\label{;zarel}
 \int _{t}^{\infty }\frac{ds}{Z_{s}^{2}}
 =\left[ 
 -\frac{1}{A_{s}}
 \right] _{t}^{\infty }=\frac{1}{A_{t}}
\end{align}
for all $t>0$ a.s., and hence 
\begin{align}\label{;ezrel}
 \eb{t}=\left( Z_{t}\int _{t}^{\infty }\frac{ds}{Z_{s}^{2}}\right) ^{-1},
 \quad \text{$t>0$, a.s.}
\end{align}
The former relation \eqref{;zarel} reveals that the process $Z$ is transient: 
$\lim \limits_{t\to \infty }Z_{t}=\infty $ a.s. 
We also recall from \cite[Theorem~1.6]{myPI} that $Z$ 
is a diffusion process in its natural filtration with the infinitesimal 
generator 
\begin{align}\label{;genz}
 \frac{1}{2}z^{2}\frac{d^{2}}{dz^{2}}+
 \left\{ 
 \frac{1}{2}z+\frac{K_{1}}{K_{0}}\left( 
 \frac{1}{z}
 \right) 
 \right\} \frac{d}{dz}. 
\end{align}
If we denote by $(Z=\{ Z_{t}\} _{t\ge 0},\{ P_{z}\} _{z\ge 0})$ 
the strong Markov family associated with \eqref{;genz} 
so that $P_{z}(Z_{0}=z)=1$, then 
Proposition~1.7 of \cite{myPI} referred to at the beginning of this 
subsection tells us that 
for every $z>0$, the functional 
\begin{align}\label{;zftnal}
 \int _{0}^{\infty }\frac{ds}{Z_{s}^{2}}
\end{align}
is distributed, under the probability measure $P_{z}$, as 
$1/\gig{0}{1}{1/z}$. Combining this fact with the expression 
\eqref{;ezrel} of geometric Brownian motion in terms of $Z$, 
leads to the statement of \pref{;pparti} in such a way that 
for $\mu >0$ and $u>0$, 
\begin{align*}
 \pr \left( 
 \eb{\tau }\in du \big| Z_{\tau }=1/\mu 
 \right) 
 &=P_{1/\mu }\biggl( 
 \mu \left( \int _{0}^{\infty }\frac{ds}{Z_{s}^{2}}\right) ^{-1}\in du
 \biggr) \\
 &=\pr \left( \mu \gig{0}{1}{\mu }\in du\right) . 
\end{align*}

\noindent 
\thetag{2} 
It would also be worthwhile mentioning that given $\mu >0$, 
by considering the process $\{ 1/Z_{t}\} _{t\ge 0}$ under 
$P_{1/\mu }$, the functional \eqref{;zftnal} is seen to be identical in law 
with the first hitting time to $0$ by a diffusion process starting from $\mu $ 
whose infinitesimal generator is given by 
\begin{align*}
 \frac{1}{2}\frac{d^{2}}{dz^{2}}
 +\left\{ 
 \frac{1}{2z}-\frac{K_{1}}{K_{0}}(z)
 \right\} \frac{d}{dz}. 
\end{align*}
A direct computation shows that for each fixed 
 $\la \in \R $, the function 
 \begin{align}\label{;eigenf}
  E_{1/\mu }\left [ 
  \exp \left(
  -\frac{\la ^{2}}{2}\int _{0}^{\infty }\frac{ds}{Z_{s}^{2}}
  \right) 
  \right] 
  =\frac{K_{0}\bigl( \mu \sqrt{1+\la ^{2}}\,\bigr) }
  {K_{0}(\mu )},\quad 
  \mu >0, 
\end{align}
solves the eigenvalue problem 
\begin{align*}
 f''(\mu )+\left\{ 
 \frac{1}{\mu }-2\frac{K_{1}}{K_{0}}(\mu )
 \right\} f'(\mu )=\la ^{2}f(\mu )
\end{align*}
associated with the above second-order differential operator. 
We refer to a more general fact by Barndorff-Nielsen 
et al.\ \cite{bbh} that any $\GIG{\nu }{a}{b}$ distribution 
with nonpositive $\nu $ is realized as the law of 
a first hitting time of some diffusion process. 
The above representation \eqref{;eigenf} follows 
readily from \eqref{;k0}. 

\noindent 
\thetag{3} We recall from \cite[Proposition~2]{yor92} 
(see also \cite[p.~43]{yorm}) that 
whenever $t>0$, the joint law of $\eb{t}$ and $A_{t}$ is given by 
\begin{align}\label{;jl}
 \pr \left( \eb{t}\in du,\,A_{t}\in dv\right) 
 =\frac{1}{uv}\exp \left( -\frac{1+u^{2}}{2v}\right) 
 \Theta _{u/v}(t)\,dudv, \quad u,v>0, 
\end{align}
where for every $r>0$, the function $\Theta _{r}(t),\,t>0$, 
is an unnormalized density of the so-called 
Hartman--Watson law with parameter $r$, which is  
characterized by 
\begin{align}\label{;ltt}
 \int _{0}^{\infty }dt\,e^{-\la ^{2}t/2}\Theta _{r}(t)
 =I_{|\la |}(r),\quad \la \in \R . 
\end{align}
Here for every $\nu \in \R $, 
the function $I_{\nu }$ denotes the modified Bessel function 
of the first kind of order $\nu $; see \cite[Section~5.7]{leb} 
for definition. As for the formula \eqref{;jl}, we also 
refer to \cite[Theorem~4.1]{mySI} as well as 
\cite[Theorem~7.5.1]{mt} for its different proof than 
the original one, based on the Sturm--Liouville theory. 
The following explicit representation 
for $\Theta _{r}(t)$ is shown in \cite{yor80}: 
\begin{align}\label{;hw}
 &\Theta _{r}(t)=\frac{r}{\sqrt{2\pi ^3t}}
 \int _{0}^\infty dy\,\exp \left( 
 \frac{\pi ^2-y^2}{2t}
 \right) \exp \left( 
 -r\cosh y
 \right) \sinh y\sin \left( \frac{\pi y}{t}\right) . 
\end{align}
Owing to \eqref{;jl}, we find that 
\begin{align}
 \pr \left( A_{t}\in dv,\,1/Z_{t}\in d\mu \right) 
 =\frac{1}{\mu }\Theta _{\mu }(t)\frac{1}{v}\exp \left( 
 -\frac{1}{2v}
 \right) \exp \left( 
 -\frac{\mu ^{2}}{2}v
 \right) dvd\mu ,\quad v,\mu >0, \label{;jlaz}
\end{align}
and that in particular, 
\begin{align}\label{;izlaw}
 \pr \left( 1/Z_{t}\in d\mu \right) 
 =\frac{2}{\mu }K_{0}(\mu )\Theta _{\mu }(t)\,d\mu ,\quad \mu >0. 
\end{align}
Indeed, for every $s>0$ and $\mu >0$, 
\begin{equation}
 \begin{split}\label{;eqjl}
 \pr \left( A_{t}\le s,\,1/Z_{t}\le \mu \right) 
 &=\int _{0}^{s}\frac{dv}{v}\int _{0}^{\mu v}\frac{du}{u}
 \exp \left( -\frac{1+u^{2}}{2v}\right) \Theta _{u/v}(t) \\
 &=\int _{0}^{\mu }\frac{du}{u}\,\Theta _{u}(t)\int _{0}^{s}\frac{dv}{v}
 \exp \left( -\frac{1}{2v}\right) \exp \left( -\frac{u^{2}}{2}v\right) , 
 \end{split}
\end{equation}
and letting $s\to \infty $ yields 
\begin{align*}
 \pr \left( 1/Z_{t}\le \mu \right) 
 =2\int _{0}^{\mu }\frac{du}{u}\Theta _{u}(t)K_{0}(u)
\end{align*}
thanks to \eqref{;k0}, where  
for the second line in \eqref{;eqjl}, we changed the variable 
$u$ into $vu$ and used Fubini's theorem. From \eqref{;jlaz} and 
\eqref{;izlaw}, we see that 
\begin{align*}
 \pr \left( 
 A_{t}\in dv\big| 1/Z_{t}=\mu 
 \right) 
 =\frac{1}{2K_{0}(\mu )}\frac{1}{v}\exp \left( -\frac{1}{2v}\right) 
 \exp \left( -\frac{\mu ^{2}}{2}v\right) dv,\quad v>0,\,\mu >0, 
\end{align*}
which is nothing but the case $\tau =t$ in \pref{;pparti} because 
conditionally on $1/Z_{t}=\mu $, we have the expression 
$A_{t}=\eb{t}/\mu $. 
\end{rem}

We proceed to the proof of \pref{;pkey}. In what follows 
the symbol $i$ stands for $\sqrt{-1}$. 

\begin{lem}\label{;lkey}
 For every $\mu >0$, $\la \ge 0$ and $\xi \in \R $, 
 it holds that for all $x\in \R $, 
 \begin{equation}
  \begin{split}\label{;eqlkey}
   &\int _{\R }dy\,e^{-\mu \cosh y}e^{-\la \cosh (x+y)}e^{i\xi \sinh (x+y)}\\
   &=\int _{0}^{\infty }\frac{dv}{v}\,\exp \left( -\frac{1}{2v}\right) 
   \exp \left( -\frac{\mu ^{2}+\la ^{2}+\xi ^{2}}{2}v \right) 
   e^{-\mu \la v\cosh x}e^{i\mu \xi v\sinh x}. 
  \end{split}
 \end{equation}
\end{lem}

\begin{proof}
 By translation, the left-hand side of \eqref{;eqlkey} is equal to 
 \begin{align*}
  \int _{\R }dy\,e^{-\mu \cosh (y-x)}e^{-\la \cosh y}e^{i\xi \sinh y}, 
 \end{align*}
 which is rewritten, by the relation 
 $\cosh (y-x)=\cosh y\cosh x-\sinh y\sinh x$, as 
 \begin{align*}
  \int _{\R }dy\,e^{-(\mu \cosh x+\la )\cosh y}e^{\mu \sinh x \sinh y}e^{i\xi \sinh y}. 
 \end{align*}
 Applying \pref{;phinge} with exponent $\mu $ therein replaced by $\mu \cosh x+\la $, 
 we see that the last expression is further rewritten as 
 \begin{align*}
  \int _{0}^{\infty }\frac{dv}{v}\,\exp \left( -\frac{1}{2v} \right) 
  \exp \left\{ -\frac{(\mu \cosh x+\la )^{2}}{2}v \right\} 
  \ex \!\left[ 
  \exp \left\{ (\mu \sinh x+i\xi ) \beta (v)\right\}
  \right] . 
 \end{align*}
 As the expectation in this integral is equal to 
 \begin{align*}
  \exp \left\{ \frac{(\mu \sinh x+i\xi )^{2}}{2}v \right\}
  =\exp \left( \frac{\mu ^{2}\sinh ^{2}x-\xi ^{2}}{2}v \right) 
  e^{i\mu \xi v\sinh x}, 
 \end{align*} 
 rearranging terms in the integral leads to the right-hand side of \eqref{;eqlkey} and 
 ends the proof. 
\end{proof}

We are in a position to prove \pref{;pkey}: as the proof shows, 
the assertion \thetag{i} is a rewriting of \eqref{;eqlkey} in terms 
of random variables. 

\begin{proof}[Proof of \pref{;pkey}]
 First we prove \thetag{i} using \lref{;lkey} prepared above. 
 To this end, we replace $\la \ge 0$ in \eqref{;eqlkey} by $|\la |$ for an arbitrary 
 $\la \in \R $. In view of \eqref{;zlaw}, we may rewrite the left-hand side of 
 \eqref{;eqlkey} as 
 \begin{align}
  &2K_{0}(\mu )\ex \!\left[ 
  \exp \left\{ -|\la |\cosh (x+\zmu ) \right\} 
  \exp \left\{ i\xi \sinh (x+\zmu ) \right\}
  \right] \notag \\
  &=2K_{0}(\mu )\ex \!\left[ 
  \exp \left\{ i\la C\cosh (x+\zmu ) \right\} 
  \exp \left\{ i\xi \sinh (x+\zmu ) \right\}
  \right] . \label{;leqlkey}
 \end{align}
 Here $C$ is independent of $\zmu $ and we have used the fact 
 that $\ex [\exp (i\al C)]=\exp (-|\al |)$ for any $\al \in \R $. 
 On the other hand, by changing the variables 
 with $v=e^{y}/\mu ,\,y\in \R $, the right-hand side of 
 \eqref{;eqlkey} with $\la $ replaced by $|\la |$, is rewritten as 
 \begin{align}
  &\int _{\R }dy\,e^{-\mu \cosh y}
  \times \exp \left( -\frac{\la ^{2}}{2}\frac{e^{y}}{\mu } -|\la |e^{y}\cosh x\right) 
  \times \exp \left( -\frac{\xi ^{2}}{2}\frac{e^{y}}{\mu } +i\xi e^{y}\sinh x\right) 
  \notag \\
  &=:\int _{\R }dy\,e^{-\mu \cosh y}\times I\times I\!I, \ \text{say}. \label{;reqlkey}
 \end{align}
 Denoting by $\beta =\{ \beta (t)\} _{t\ge 0}$ a Brownian motion 
 independent of $C$, we have 
 \begin{align*}
  I&=\ex \!\left[ 
  \exp \left\{ i\la \left( Ce^{y}\cosh x+\beta (e^{y}/\mu ) \right) \right\}
  \right] \\
 \intertext{and likewise}
  I\!I&=\ex \!\left[ 
  \exp \left\{ i\xi \left( e^{y}\sinh x+\beta (e^{y}/\mu ) \right) \right\}
  \right] . 
 \end{align*}
 Plugging these into \eqref{;reqlkey} and noting 
 \eqref{;zlaw} again, we conclude that the righ-hand side of \eqref{;eqlkey} 
 admits the expression 
 \begin{align*}
  2K_{0}(\mu )\ex \!\left[ 
  \exp \bigl\{ 
  i\la \bigl( Ce^{\zmu }\cosh x+\hbe (e^{\zmu }/\mu ) \bigr) 
  \bigr\}
  \exp \left\{ 
  i\xi \left( e^{\zmu }\sinh x+\beta (e^{\zmu }/\mu ) \right) 
  \right\} 
  \right] , 
 \end{align*}
 where $\hbe =\{ \hbe (t)\} _{t\ge 0}$ is another Brownian 
 motion, and $\zmu $, $\beta $, $\hbe $ and $C$ are 
 independent. Since the last expression agrees with 
 \eqref{;leqlkey} for any $\la ,\xi \in \R $, we obtain  
 the identity \eqref{;eqpkey1} by successive use of the injectivity 
 of Fourier transform. 
 
 To prove the equivalent expression 
 \eqref{;eqpkey2}, it suffices to note the fact that for two 
 independent Brownian motions $\beta $ and 
 $\hB =\{ \hB _{t}\} _{t\ge 0}$, there holds the identity in law 
 $\beta (\tau _{a}(\hB ))\eqd aC$ for every $a\in \R $: 
 \begin{equation}\label{;cauchy}
  \begin{split}
   \ex \!\left[ 
   \exp \left\{ i\xi  \beta (\tau _{a}(\hB ))\right\}
   \right] 
   &=\ex \!\left[ 
   \exp \left\{  -\frac{\xi ^{2}}{2}\tau _{a}(\hB )\right\}
   \right] \\
   &=\exp (-\left| \xi \right| \!\left| a\right| ), 
  \end{split}
 \end{equation}
 where $\xi \in \R $ is arbitrary and the second line is due to 
 \eqref{;lt}. 
 Using this well-known fact, we may rephrase the identity 
 \eqref{;eqpkey1} as 
 \begin{align*}
  \left( 
  e^{\zmu }\sinh x+\beta (e^{\zmu }/\mu ),\,
  \hbe \bigl( 
  \tau _{e^{\zmu }\cosh x}(\hB )+e^{\zmu }/\mu 
  \bigr) 
  \right) 
  \eqd 
  \left( 
  \sinh (x+\zmu ),\,
  \hbe \bigl( 
  \tau _{\cosh (x+\zmu )}(\hB )
  \bigr) 
  \right) , 
 \end{align*}
 where $\hbe =\{ \hbe (t)\} _{t\ge 0}$ is a Brownian motion 
 independent of $\zmu $, $\beta $ and $\hB $. Taking the 
 joint Fourier transform on both sides leads to coincidence 
 of joint Fourier--Laplace transforms in such a way that for 
 any $\la ,\xi \in \R $, 
 \begin{align*}
  &\ex \!\left[ 
  \exp \left\{ 
  i\la \left( e^{\zmu }\sinh x+\beta (e^{\zmu }/\mu )\right) 
  \right\} 
  \exp \left\{ 
  -\frac{\xi ^{2}}{2}\left( 
  \tau _{e^{\zmu }\cosh x}(\hB )+\frac{e^{\zmu }}{\mu} 
  \right) 
  \right\} 
  \right] \\
  &=\ex \!\left[ 
  \exp \left\{ i\la \sinh (x+\zmu )\right\} 
  \exp \left\{ 
  -\frac{\xi ^{2}}{2}
  \tau _{\cosh (x+\zmu )}(\hB )
  \right\} 
  \right] , 
 \end{align*}
 which proves the identity \eqref{;eqpkey2} owing to the 
 injectivity of Fourier and Laplace transforms. From the above 
 argument it is obvious that \eqref{;eqpkey2} implies \eqref{;eqpkey1} 
 conversely. 
 
  We turn to the proof of \thetag{ii}. Let $f:\R \to \R $ be bounded 
 and measurable, 
 and take $\la \in \R $ arbitrarily. For a Brownian motion 
 $\beta =\{ \beta (t)\} _{t\ge 0}$ independent of $\zmu $, we compute 
 \begin{equation}\label{;lastexpr}
  \begin{split}
  &2K_{0}(\mu )\ex \!\left[ 
  f\bigl( e^{\zmu }\sinh x+\beta (e^{\zmu }/\mu )\bigr) 
  \exp \left( -\frac{\la ^{2}}{2}\frac{e^{\zmu }}{\mu } \right)
  \right] \\
  &=\int _{\R }dy\,e^{-\mu \cosh y}\exp \left(  
  -\frac{\la ^{2}}{2}\frac{e^{y}}{\mu }\right) 
  \ex \!\left[ 
  f\bigl( e^{y}\sinh x+\beta (e^{y}/\mu )\bigr) 
  \right] \\
  &=\int _{0}^{\infty }\frac{dv}{v}\,
  \exp \left( -\frac{1}{2v}\right) 
  \exp \left( -\frac{\mu ^{2}\cosh ^{2}x+\la ^{2}}{2}v \right) \\
  &\hspace{60.5mm}\times 
  \ex \!\left[ f(\beta (v))e^{\mu \sinh x\cdot \beta (v)}\right] \\
  &=\int _{\R }dy\,
  e^{
  -(\sqrt{\mu ^{2}\cosh ^{2}x+\la ^{2}}\cosh y-\mu \sinh x \sinh y) 
  }f(\sinh y), 
  \end{split}
 \end{equation}
 where we used the change of the variables with $e^{y}/\mu =v$ 
 as well as the Cameron--Martin formula applied to 
 $\mu \sinh x\cdot v+\beta (v)$ for the 
 second equality and \pref{;phinge} for the third. 
 Thanks to \eqref{;lt}, we may rewrite 
 \begin{align*}
  &e^{
  -(\sqrt{\mu ^{2}\cosh ^{2}x+\la ^{2}}\cosh y-\mu \sinh x \sinh y) 
  }\\
  &=e^{-\mu \cosh (y-x)}
  e^{-(\sqrt{\mu ^{2}\cosh ^{2}x+\la ^{2}}-\mu \cosh x)
  \cosh y}\\
  &=e^{-\mu \cosh (y-x)}
  \ex \!\left[ 
  \exp \left\{ -\frac{\la ^{2}}{2}\tau _{\cosh y}(\db{\mu \cosh x})
  \right\} 
  \right] . 
 \end{align*}
 We plug the last expression into the rightmost side 
 of \eqref{;lastexpr}. After translating the variable $y$ by $x$,  
 we use Fubini's theorem to arrive at the identity 
 \begin{align*}
  &\ex \!\left[ 
  f\bigl( e^{\zmu }\sinh x+\beta (e^{\zmu }/\mu )\bigr) 
  \exp \left( -\frac{\la ^{2}}{2}\frac{e^{\zmu }}{\mu } \right)
  \right] \\
  &=\ex \!\left[ 
  f(\sinh (x+\zmu ))
  \exp \left\{ -\frac{\la ^{2}}{2}
  \tau _{\cosh (x+\zmu )} (\hB ^{(\mu \cosh x)})
  \right\} 
  \right] . 
 \end{align*}
 Here on the right-hand side, $\hB $ denotes a Brownian motion independent 
 of $\zmu $. As $f$ and $\la $ are arbitrary, we obtain the 
 identity \eqref{;eqpkey3} as claimed. Proof of \pref{;pkey} 
 is complete. 
\end{proof}

We close this subsection by listing some facts deduced from \pref{;pkey}. 

\begin{rem}\label{;rpkey}
\thetag{1} If we let $\mu \to \infty $, the identity \eqref{;eqpkey1} 
remains valid in the sense that as 
two-dimensional random variables, both sides of \eqref{;eqpkey1} 
converge in law to $(\sinh x,C\cosh x)$ by \pref{;pconvlaw}. 

\noindent 
\thetag{2} It follows readily from \pref{;pconvlaw} that 
$\sinh \zmu $ converges in law to $0$ 
as $\mu \to \infty $, which may be strengthened as 
convergence in law of $\sqrt{\mu }\sinh \zmu $ to the standard normal 
distribution. Indeed, by the scaling property of Brownian motion and by 
\eqref{;eqpkey1} (or \eqref{;proto1}), 
\begin{align*}
 \beta (e^{\zmu })&\eqd \sqrt{\mu }\beta (e^{\zmu }/\mu )\\
 &\eqd \sqrt{\mu }\sinh \zmu , 
\end{align*}
the leftmost side of which converges in law to $\beta (1)$ as 
$\mu \to \infty $ thanks to \pref{;pconvlaw}. We may 
rephrase this convergence in law as that of 
$
(\sqrt{\mu }/2)\bigl( 
\gig{0}{\sqrt{\mu }}{\sqrt{\mu }}-1/\gig{0}{\sqrt{\mu }}{\sqrt{\mu }}
\bigr) 
$ in view of \eqref{;idengig}. 
The same convergence also holds true for 
any $\GIG{\nu }{\sqrt{\mu }}{\sqrt{\mu }}$ laws; in fact, a simple change 
of the variables shows that for every $\nu \in \R $ and for any bounded measurable function $f$ on $\R $, 
\begin{align*}
 &\ex \Bigl[ 
 f\Bigl( 
 (\sqrt{\mu }/2)\bigl( 
\gig{\nu }{\sqrt{\mu }}{\sqrt{\mu }}-1/\gig{\nu }{\sqrt{\mu }}{\sqrt{\mu }}
\bigr) 
 \Bigr) 
 \Bigr] \\
 &=\frac{1}{2\sqrt{\mu }e^{\mu }K_{\nu }(\mu )}
 \int _{\R }\frac{dx}{\sqrt{1+x^{2}/\mu }}\,
 \left( 
 \frac{x}{\sqrt{\mu }}
 +\sqrt{1+\frac{x^{2}}{\mu }}
 \right) ^{\nu }
 \!\exp \left( 
 -\frac{x^{2}}{1+\sqrt{1+x^{2}/\mu }}
 \right) \!f(x), 
\end{align*}
which converges to $\ex [f(\beta (1))]$ 
as $\mu \to \infty $ due to \eqref{;klim}. 
In \cite{jor}, a similar computation reveals the convergence in law of  
$
\sqrt{\mu }\bigl( \gig{\nu }{\sqrt{\mu }}{\sqrt{\mu }}-1\bigr) 
$
to the standard normal distribution, 
from which the same convergence of 
$\sqrt{\mu }\log \bigl( \gig{\nu }{\sqrt{\mu }}{\sqrt{\mu }}\bigr) $ 
is also deduced; see Equations~\thetag{3.10} and \thetag{3.11} 
of the above-cited reference. 

\noindent 
\thetag{3} As for each side of the identity \eqref{;eqpkey1}, 
if we divide the second coordinate by the square root of the sum of 
1 and square of the first coordinate, then we obtain the 
identity in law: 
\begin{align*}
 \frac{
 Ce^{z_{\mu }}\cosh x+\hbe (e^{z_{\mu }}/\mu )
 }
 {
 \sqrt{
 1+\left( 
 e^{z_{\mu }}\sinh x+\beta (e^{z_{\mu }}/\mu )
 \right) ^{2}
 }
 }
 \eqd C, 
\end{align*}
which holds for every fixed $\mu >0$ and $x\in \R $. 
The above identity may also be deduced from \eqref{;invcauchy} 
by taking $\tau =\tau _{1/\mu }(Z)$, which is finite a.s.\ 
because of the transience of $Z$ noted in 
\rref{;rpparti}\,\thetag{1}.  

\noindent 
\thetag{4} For any $x,y\in \R $, there holds the identity in law 
\begin{align*}
 \sinh (y+\zmu )
 \eqd 
 \sinh (x+\zmu )+\mu (\sinh y-\sinh x)
 \tau _{\cosh (x+\zmu )}(\hB ^{(\mu \cosh x)})
\end{align*}
with $\hB $ a Brownian motion independent of $\zmu $ 
on the right-hand side; indeed, 
\begin{align*}
 \sinh (y+\zmu )&\eqd e^{\zmu }\sinh y+\beta (e^{\zmu }/\mu )\\
 &=e^{\zmu }(\sinh y-\sinh x)+e^{\zmu }\sinh x+\beta (e^{\zmu }/\mu )\\
 &\eqd \mu \tau _{\cosh (x+\zmu )}(\hB ^{(\mu \cosh x)})
 \cdot (\sinh y-\sinh x)+\sinh (x+\zmu ), 
\end{align*}
where the third line is due to \eqref{;eqpkey3}.  
\end{rem}

\subsection{Proof of \tref{;tm3}}\label{;prftm3}

In this subsection we prove \tref{;tm3}, whose assertion 
is nothing but a rewriting of the following one: 
as in the statement of the theorem, $\ve $ denotes a 
Rademacher variable. 

\begin{prop}\label{;pcauchy}
 Suppose that $C$ and $\ve $ are independent. Then it holds that: 
 \begin{enumerate}[(i)]{}
  \item for every $x,y\in \R $, 
  \begin{align*}
   \sinh \left( \argsh (C\cosh x)+y\ve \right) 
   \eqd C\cosh (x+y\ve );  
  \end{align*}
  \item for every $a\in [-1,1]$ and $y\in \R $, 
  \begin{align*}
   \sinh \left( \argsh (aC)+y\ve \right) 
   \eqd aC\cosh y+\sqrt{1-a^{2}}\sinh (y\ve ). 
  \end{align*}
  \end{enumerate}
\end{prop}

Since we will use \lref{;lkey} in the case $\la =0$ below, we repeat it 
here for the reader's convenience, putting $\la =0$: 
for every $\mu >0$ and $\xi \in \R $, it holds that 
\begin{equation}\label{;eqlkeyd}
 \int _{\R }dy\,e^{-\mu \cosh y}e^{i\xi \sinh (x+y)}
 =\int _{0}^{\infty }\frac{dv}{v}\,\exp \left( -\frac{1}{2v}\right) 
 \exp \left( -\frac{\mu ^{2}+\xi ^{2}}{2}v \right) e^{i\mu \xi v\sinh x}
\end{equation}
 for all $x\in \R $. 
 
\begin{proof}[Proof of \pref{;pcauchy}]
We start with the proof of \thetag{i}. We replace $x$ in \eqref{;eqlkeyd} by 
$\argsh (C\cosh x)$ and take the expectation on both sides with respect to 
$C$. Then by Fubini's theorem, we have 
\begin{equation}\label{;eqcauchy1}
 \begin{split}
 &\int _{\R }dy\,e^{-\mu \cosh y}\ex \!\left[ 
 \exp \left\{ 
 i\xi \sinh \left( 
 \argsh (C\cosh x)+y
 \right)
 \right\}
 \right] \\
 &=\int _{0}^{\infty }\frac{dv}{v}\,
 \exp \left( 
 -\frac{1}{2v}
 \right)
 \exp 
 \left( -\frac{\mu ^{2}+\xi ^{2}+2\mu |\xi |\cosh x}{2}v
 \right) 
 \end{split}
\end{equation}
thanks to the identity 
$
\ex \!\left[ 
e^{i\mu \xi vC\cosh x}
\right] =\exp (-\mu |\xi |v\cosh x)
$ 
for $v>0$. Rewriting 
\begin{align*}
 \exp 
 \left( -\frac{\mu ^{2}+\xi ^{2}+2\mu |\xi |\cosh x}{2}v
 \right) 
 &=\exp \left\{ 
 -\frac{(\mu +|\xi |\cosh x)^{2}}{2}v+\frac{\xi ^{2}}{2}v\sinh ^{2}x
 \right\} \\
 &=\exp \left\{ 
 -\frac{(\mu +|\xi |\cosh x)^{2}}{2}v
 \right\} \ex \!\left[ 
 e^{-|\xi |\beta (v)\sinh x}
 \right] 
\end{align*}
with $\beta $ a Brownian motion, we see from \pref{;phinge} that 
the right-hand side of \eqref{;eqcauchy1} is equal to 
\begin{align*}
 \int _{\R }dy\,e^{-(\mu +|\xi |\cosh x)\cosh y}
 e^{-|\xi |\sinh x\sinh y}
 =\int _{\R }dy\,e^{-\mu \cosh y}e^{-|\xi |\cosh (x+y)}. 
\end{align*}
Replacing $e^{-|\xi |\cosh (x+y)}$ in the last integral by  
$\ex \!\left[ 
\exp \{ i\xi C\cosh (x+y)\} 
\right] $, 
we obtain from \eqref{;eqcauchy1} the identity 
\begin{equation}\label{;eqpcauchy2}
 \begin{split}
 &\int _{\R }dy\,e^{-\mu \cosh y}\ex \!\left[ 
 \exp \left\{ 
 i\xi \sinh \left( 
 \argsh (C\cosh x)+y
 \right)
 \right\}
 \right] \\
 &=\int _{\R }dy\,e^{-\mu \cosh y}
 \ex \!\left[ 
 \exp \{ i\xi C\cosh (x+y)\} 
 \right] , 
 \end{split}
\end{equation}
which holds for every $\mu >0$, $\xi \in \R $ and $x\in \R $. 
Now we observe that for any bounded measurable function 
$f:\R \to \R $, 
\begin{align*}
 \int _{\R }dy\,e^{-\mu \cosh y}f(y)
 &=2\int _{0}^{\infty }dy\,e^{-\mu \cosh y}\ex [f(\ve y)]\\
 &=2\int _{1}^{\infty }\frac{du}{\sqrt{u^{2}-1}}\,e^{-\mu u}
 \ex \!\left[ 
 f(\ve \argch u)
 \right] , 
\end{align*}
where the second line is due to change of the variables with 
$y=\argch u,\,u>1$. Here $\argch u=\log \left( u+\sqrt{u^{2}-1}\right) $. 
Thanks to the above observation as well as to the injectivity of Laplace 
transform, we conclude from \eqref{;eqpcauchy2} that  
\begin{align*}
 &\ex \!\left[ 
 \exp \left\{ 
 i\xi \sinh \left( 
 \argsh (C\cosh x)+\ve \argch u
 \right)
 \right\}
 \right] \\
 &=\ex \!\left[ 
 \exp \left\{ i\xi C\cosh \left( x+\ve \argch u\right) \right\} 
 \right] 
\end{align*}
for any $u\ge 1$, where $\ve $ is assumed to be independent 
of $C$ on both sides. The assertion \thetag{i} of the proposition then 
follows readily since $\xi $ is arbitrary and $\ve $ is symmetric. 

Proof of the assertion \thetag{ii} is done similarly to 
the above. For each fixed $a\in [-1,1]$, we replace $x$ in 
\eqref{;eqlkeyd} by $\argsh (aC)$ and take the expectation on 
both sides to get the equality 
\begin{equation}\label{;eqpcauchy3}
 \begin{split}
 &\int _{\R }dy\,e^{-\mu \cosh y}\ex \!\left[ 
 \exp \left\{ 
 i\xi \sinh \left( 
 \argsh (aC)+y
 \right)
 \right\}
 \right] \\
 &=\int _{0}^{\infty }\frac{dv}{v}\,
 \exp \left( 
 -\frac{1}{2v}
 \right)
 \exp \left\{ 
 -\frac{(\mu +|a\xi |)^{2}}{2}v
 \right\} 
 \exp \left\{ 
 -\frac{\xi ^{2}}{2}(1-a^{2})v
 \right\} . 
 \end{split}
\end{equation}
 By noting that 
 \begin{align*}
  \exp \left\{ 
 -\frac{\xi ^{2}}{2}(1-a^{2})v
 \right\} 
 =\ex \!\left[ 
 \exp \left\{ i\xi \sqrt{1-a^{2}}\beta (v)\right\}
 \right] 
 \end{align*}
 with $\beta $ a Brownian motion, and by using \pref{;phinge}, 
 the right-hand side of \eqref{;eqpcauchy3} is rewritten as 
 \begin{align*}
  &\int _{\R }dy\,e^{-(\mu +|a\xi |)\cosh y}e^{i\xi \sqrt{1-a^{2}}
  \sinh y}\\
  &=\int _{\R }dy\,e^{-\mu \cosh y}
  \ex \!\left[ 
  \exp \left\{ 
  i\xi \left( 
  aC\cosh y+\sqrt{1-a^{2}}\sinh y
  \right) 
  \right\} 
  \right] . 
 \end{align*}
 Rest of the proof proceeds in the same way as in the latter half of 
 the proof of \thetag{i} by noting the symmetry of the 
 function $\cosh y,\,y\in \R $. 
\end{proof}

We are ready to prove \tref{;tm3}. 

\begin{proof}[Proof of \tref{;tm3}]
 The assertion \thetag{i} of \pref{;pcauchy} entails that 
 by symmetry of $\ve $, 
 \begin{align*}
  aC\cosh y+\sqrt{1+a^{2}C^{2}}\,\ve \sinh y 
  \eqd C\left( a\cosh y+\sqrt{a^{2}-1}\,\ve \sinh y\right) 
 \end{align*}
 for any $y\in \R $. Here we write $a=\cosh x\ge 1$. 
 Dividing both sides by $\cosh y$, we see that the assertion 
 \thetag{i} of the theorem holds true for any $\theta \in (-1,1)$, 
 and hence for any $\theta \in [-1,1]$ since it is 
 clear that both sides of the claimed identity are continuous 
 in $\theta $ with respect to the topology of 
 weak convergence. 
 Extension to the case $a\le -1$ is done by noting that 
 $(C,-\ve )$ and $(-C,\ve )$, as well as $(-C,-\ve )$, 
 are identical in law with $(C,\ve )$ due to independence 
 of $C$ and $\ve $. This completes the proof of the 
 assertion~\thetag{i}. We omit the proof of \thetag{ii} 
 because it proceeds quite similarly. 
\end{proof}

We close this section with a remark on \tref{;tm3} as well as 
on the integral identity \eqref{;eqlkeyd} used in the 
proof of \pref{;pcauchy}. 

\begin{rem}\label{;rtm3}
\thetag{1} For every fixed $\theta \in [-1,1]$, 
\tref{;tm3} entails that the characteristic function 
of $aC+\theta \sqrt{1+a^{2}C^{2}}\,\ve $ is given by 
\begin{align*}
 \exp (-|a\xi |)\cosh \left( \theta \xi \sqrt{a^{2}-1}\right) , 
 \quad \xi \in \R , 
\end{align*}
in the case $|a|\ge 1$ while  
it is given by 
\begin{align*}
 \exp (-|a\xi |)\cos \left( \theta \xi \sqrt{1-a^{2}}\right) , 
 \quad \xi \in \R , 
\end{align*}
in the case $|a|\le 1$. To verify the expression 
in the former case, it suffices to deduce from \tref{;tm3}\,\thetag{i} 
that for any $\xi \in \R $, 
\begin{align*}
 \ex \!\left[ 
 \exp \left\{ 
 i\xi \bigl( 
 aC+\theta \sqrt{1+a^{2}C^{2}}\,\ve 
 \bigr) 
 \right\} 
 \right] 
 =\ex \!\left[ 
 \exp \left\{ 
 -|\xi|\bigl|  
 a+\theta \sqrt{a^{2}-1}\,\ve 
 \bigr| 
 \right\} 
 \right] , 
\end{align*}
and that since $|\theta |\le 1$ and $|\ve |=1$ a.s., 
\begin{align*}
 \bigl|  
 a+\theta \sqrt{a^{2}-1}\,\ve 
 \bigr| =|a|+\theta \sqrt{a^{2}-1}\,\ve \,\mathrm{sgn}\,a 
 \quad \text{a.s.}, 
\end{align*}
where $\mathrm{sgn}\,a$ is the signature of $a$. 
The above expressions of 
characteristic functions in two cases suggest that 
by means of analytic continuation, 
the assertion~\thetag{ii} of the theorem may be seen 
as a consequence of the assertion~\thetag{i} or vice versa. 

\noindent 
\thetag{2} It is of interest to observe from \eqref{;eqlkeyd} 
that if we take $\xi $ positive, then the multivariate 
function 
\begin{align*}
 \int _{\R }dy\,e^{-\mu \cosh y}e^{i\xi \sinh (x+y)},\quad 
 \mu ,\xi >0,\ x\in \R , 
\end{align*}
is symmetric with respect to $\mu $ and $\xi $. 
As will be seen in \ssref{;sym} below, this fact yields 
the following symmetry concerning the laws of 
functionals of the form $e^{2B_{t}}v+A_{t},\,v>0$, for 
every $t>0$: 
\begin{align*}
 &\frac{1}{v}\exp \left( -\frac{1}{2v}\right) 
 \pr \left( e^{2B_{t}}v+A_{t}\in du\right) dv\\
 &=\frac{1}{u}\exp \left( -\frac{1}{2u}\right) 
 \pr \left( e^{2B_{t}}u+A_{t}\in dv\right) du,
 \quad u,v>0. 
\end{align*}
We may replace $t$ by any positive and finite 
stopping time $\tau $ of the process $Z$. 
\end{rem}

\section{Some related results}\label{;related} 

In this section, we present several results relevant to 
our discussions developed in the previous section. 

\subsection{Derivations of \eqref{;eqtm1} in part and Dufresne's identity}\label{;partial}

This subsection is concerned with partial derivations of 
the identity \eqref{;eqtm1} in \tref{;tm1}, one of which is 
to be applied to a proof of an identity in law due to Dufresne \cite{duf}. 

For each $\nu \in \R $, we set 
\begin{align*}
 \da{\nu }_{t}:=\int _{0}^{t}e^{2\db{\nu }_{s}}ds,\quad t\ge 0, 
\end{align*}
where $\db{\nu }$ is the Brownian motion with drift $\nu $. 
Dufresne's identity (\cite[Proposition~4.4.4\,(b)]{duf}) asserts 
that when $\nu >0$, the perpetual integral 
$
\da{-\nu }_{\infty }
=\int _{0}^{\infty }e^{2\db{-\nu }_{s}}ds
$ 
is identical in law with the reciprocal of twice of a 
gamma variable with parameter $\nu $: 
\begin{align}
 \da{-\nu }_{\infty }
 &\eqd \frac{1}{2\ga _{\nu }}, \label{;duf} \\
 \pr (\ga _{\nu }\in dv)
 &=\frac{1}{\Gamma (\nu )}v^{\nu -1}e^{-v}\,dv,\quad v>0, 
 \label{;pdfg} 
\end{align}
where $\Gamma $ is the gamma function; 
we refer to Theorem~6.2 and Proposition~6.3 in \cite{mySI} 
for different proofs of \eqref{;duf} than the original one. 

In this subsection, we use the invariance formulae for 
Cauchy variable in the form \pref{;pcauchy} to see 
that the identity \eqref{;eqtm1} is recovered partly in two ways; 
we then apply one of those partial identities to give another proof 
of Dufresne's identity \eqref{;duf}. Our argument to be developed 
here does not use \pref{;pparti} and relies only on \pref{;pcauchy} 
and the identity \eqref{;boug2}, which makes the contents 
of this subsection completely self-contained. 

In the proposition exhibited below, the first identity 
\eqref{;eqppar1} treats the one between second coordinates 
in \eqref{;eqtm1} when $\tau =t$, and the second 
identity \eqref{;eqppar2} treats the case $x=0$ and 
$\tau =t$ in \eqref{;eqtm1} with third coordinates dropped. 

\begin{prop}\label{;ppar}
Under the same setting as in \tref{;tm1}, it holds that for 
every fixed $t>0$ and $x\in \R $, 
\begin{align}
 C\eb{t}\cosh x+\hbe (A_{t})
 &\eqd 
 C\cosh (x+B_{t}), \label{;eqppar1}
\intertext{and that for every fixed $t>0$,}
 \bigl( 
 \beta (A_{t}),\,C\eb{t}+\hbe (A_{t})
 \bigr) 
 &\eqd \left( 
 \sinh B_{t},\,C\cosh B_{t}
 \right) . \label{;eqppar2}
\end{align}
\end{prop}

\begin{proof}
 For three independent elements $B$, $C$ and a Rademacher variable $\ve $, 
 we substitute $y$ in \pref{;pcauchy}\,\thetag{i} by 
 $B_{t}$ to obtain 
 \begin{align*}
  \sinh \left( \argsh (C\cosh x)+B_{t}\right) 
  \eqd C\cosh (x+B_{t}), 
 \end{align*}
 where we used the fact that $\ve B_{t}\eqd B_{t}$ 
 by independence of $B$ and $\ve $ and by symmetry 
 of $B$. In virtue of the identity \eqref{;boug2}, 
 the left-hand side of the above identity is identical in law 
 with 
 \begin{align*}
  C\eb{t}\cosh x+\hbe (A_{t}), 
 \end{align*}
 which shows \eqref{;eqppar1}. 
 
 To prove \eqref{;eqppar2}, 
 we repeat the same argument as above to obtain from 
 \pref{;pcauchy}\,\thetag{ii}, 
 \begin{align*}
  aC\eb{t}+\beta (A_{t})\eqd 
  aC\cosh B_{t}+\sqrt{1-a^{2}}\sinh B_{t}
 \end{align*}
 for every $a\in [-1,1]$, where on the left-hand side, 
 $\beta $ is a Brownian motion independent of $B$ and $C$. 
 Taking the Fourier transform 
 on both sides, we have for any $\xi \in \R $, 
 \begin{align}\label{;jlt1}
  \ex \!\left[ 
  \exp \left( 
  -|a\xi |\eb{t}-\frac{\xi ^{2}}{2}A_{t}
  \right) 
  \right] 
  =\ex \!\left[ 
  \exp \left( 
  ia\xi C\cosh B_{t}
  \right) 
  \exp \bigl( 
  i\xi \sqrt{1-a^{2}}\sinh B_{t}
  \bigr) 
  \right] . 
 \end{align}
 We may express the left-hand side of \eqref{;jlt1} as 
 \begin{align*}
  \ex \!\left[ 
  \exp \bigl\{ 
  ia\xi \bigl( 
  C\eb{t}+\hbe (A_{t})
  \bigr) 
  \bigr\} 
  \exp \bigl\{ 
  i\xi \sqrt{1-a^{2}}\beta (A_{t})
  \bigr\} 
  \right] 
 \end{align*}
 in terms of the pair of random variables on the 
 left-hand side of \eqref{;eqppar2}. Since the last 
 two expressions agree for any $a$ and $\xi $, the second 
 identity \eqref{;eqppar2} also follows.  
\end{proof}

Using the former identity \eqref{;eqppar1} in \pref{;ppar}, 
we are going to prove 

\begin{prop}\label{;pduf}
 Let $\nu >0$ and fix $t>0$. Then it holds that 
 \begin{align}\label{;eqpduf}
  \frac{e^{2\db{-\nu }_{t}}}{2\ga _{\nu }}
  +\da{-\nu }_{t}\eqd \frac{1}{2\ga _{\nu }}, 
\end{align}
where on the left-hand side, the gamma variable 
$\ga _{\nu }$ with parameter $\nu $ is independent of 
$B$. In particular, letting $t\to \infty $ on the 
left-hand side leads to \eqref{;duf}. 
\end{prop}

\begin{rem}\label{;reqpduf}
 It should be noted that Dufresne's identity \eqref{;duf} also implies 
 \eqref{;eqpduf}; indeed, for every $t>0$, 
 \begin{align*}
  \da{-\nu }_{\infty }&=\da{-\nu }_{t}
  +\int _{t}^{\infty }e^{2\db{-\nu }_{s}}ds\\
  &=\da{-\nu }_{t}+e^{2\db{-\nu }_{t}}\!\!
  \int _{0}^{\infty }e^{2(\db{-\nu }_{s+t}-\db{-\nu }_{t})}ds, 
 \end{align*}
 the last integral being independent of $\db{-\nu }_{s},\,0\le s\le t$, 
 and by \eqref{;duf}, distributed as $1/(2\ga _{\nu })$. 
\end{rem}

Notice that by the same reasoning as used in the proof of the identity 
\eqref{;eqpkey2} in \pref{;pkey}, the identity \eqref{;eqppar1} 
admits the equivalent expression 
\begin{align}\label{;eqduf}
 \tau _{\eb{t}\cosh x}(\hB )+A_{t}
 \eqd 
 \tau _{\cosh (x+B_{t})}(\hB )
\end{align}
as in \eqref{;eqtm1d}. Here $\hB $ is a Brownian motion independent of 
$B$ on both sides. We start the proof of \pref{;pduf} from \eqref{;eqduf}. 
We also observe from \eqref{;pdfg} and another integral representation 
\eqref{;intk2} of the Macdonald function that 
when $\nu >0$, we may put that representation into the form 
\begin{align}\label{;intk2d}
 K_{\nu }(z)=2^{\nu -1}\Gamma (\nu )z^{-\nu }
 \ex \!\left[ 
 \exp \left( 
 -\frac{z^{2}}{4\ga _{\nu }}
 \right) 
 \right] ,\quad z>0. 
\end{align}

\begin{proof}[Proof of \pref{;pduf}]
 Recalling \eqref{;lt} and replacing $\la $ therein by 
 every $z>0$, we take the Laplace transform on both sides of 
 \eqref{;eqduf} to get 
 \begin{align}\label{;jlt2}
  \ex \!\left[ 
  \exp \left( 
  -z\eb{t}\cosh x-\frac{z^{2}}{2}A_{t}
  \right) 
  \right] 
  =\ex \!\left[ 
  \exp \left\{ 
  -z\cosh (x+B_{t})
  \right\} 
  \right] . 
 \end{align}
 For each fixed $\nu \in \R $, we integrate both sides over $\R $ 
 with respect to $(1/2)e^{-\nu x}\,dx$. Then by Fubini's theorem and 
 by the integral representation \eqref{;intk1} of $K_{\nu }$, 
 the left-hand side becomes 
 \begin{align}\label{;eqpduf1}
  \ex \!\left[ 
  K_{\nu }(z\eb{t})\exp \left( 
  -\frac{z^{2}}{2}A_{t}
  \right) 
  \right] . 
 \end{align}
 As for the right-hand side, by Fubini's theorem, we compute 
 \begin{align}
  \frac{1}{2}\ex \!\left[ 
  \int _{\R }dx\,e^{-\nu x}\exp \left\{ 
  -z\cosh (x+B_{t})
  \right\} 
  \right] 
  &=\frac{1}{2}\ex \!\left[ 
  \int _{\R }dx\,e^{-\nu (x-B_{t})}\exp (-z\cosh x)
  \right] \notag \\
  &=\frac{1}{2}\int _{\R }dx\,e^{-z\cosh x}e^{-\nu x}\times 
  \ex \!\left[ 
  e^{\nu B_{t}}
  \right] \notag \\
  &=K_{\nu }(z)e^{\nu ^{2}t/2}. \label{;eqpduf2}
 \end{align}
 We let $\nu >0$ hereafter. Assuming that 
 $\ga _{\nu }$ is independent of $B$, we now use \eqref{;intk2d} 
 to rewrite \eqref{;eqpduf1} as 
 \begin{align*}
  &2^{\nu -1}\Gamma (\nu )z^{-\nu }
  \ex \!\left[ 
  e^{-\nu B_{t}}
  \exp \left\{ 
  -\frac{z^{2}}{2}\left( 
  \frac{e^{2B_{t}}}{2\ga _{\nu }}+A_{t}
  \right) 
  \right\} 
  \right] \\
  &=2^{\nu -1}\Gamma (\nu )z^{-\nu }
  \ex \!\left[ 
  \exp \left\{ 
  -\frac{z^{2}}{2}\left( 
  \frac{e^{2\db{-\nu }_{t}}}{2\ga _{\nu }}+\da{-\nu }_{t}
  \right) 
  \right\} 
  \right] e^{\nu ^{2}t/2}. 
 \end{align*}
 Here we used the Cameron--Martin formula for the equality. Since 
 the last expression agrees with \eqref{;eqpduf2}, we have for any 
 $z>0$, 
 \begin{align}\label{;eqpduf3}
  \ex \!\left[ 
  \exp \left\{ 
  -\frac{z^{2}}{2}\left( 
  \frac{e^{2\db{-\nu }_{t}}}{2\ga _{\nu }}+\da{-\nu }_{t}
  \right) 
  \right\} 
  \right] 
  =
  \ex \!\left[ 
 \exp \left( 
 -\frac{z^{2}}{2}\cdot \frac{1}{2\ga _{\nu }}
 \right) 
 \right] 
 \end{align}
 thanks to \eqref{;intk2d}. Therefore the identity \eqref{;eqpduf} 
 is proven by the injectivity of Laplace transform. 
 Letting $t\to \infty $ on the left-hand side of 
 \eqref{;eqpduf3} also proves \eqref{;duf} 
 by the bounded convergence theorem. 
\end{proof}

\begin{rem}\label{;requiv}
\thetag{1} Recalling the well-known fact, 
as readily seen from \eqref{;idenst} and 
the identity $B_{1}^{2}\eqd 2\ga _{1/2}$, that 
\begin{align*}
 \tau _{a}(B)\eqd 
 \frac{a^{2}}{2\ga _{1/2}}
\end{align*}
for every $a\in \R $, we may rephrase \eqref{;eqduf} as 
\begin{align}\label{;eqdufd}
 \frac{e^{2B_{t}}\cosh ^{2}x}{2\ga _{1/2}}+A_{t}
 \eqd 
 \frac{\cosh ^{2}(x+B_{t})}{2\ga _{1/2}}, 
\end{align}
which holds for every fixed $t>0$ and $x\in \R $. 
Here $B$ and $\ga _{1/2}$ are independent on both sides. 
For each fixed $t$, the family of identities \eqref{;eqdufd} 
indexed by $x\in \R $, may be regarded as an equivalent expression 
for that of identities \eqref{;eqpduf} indexed by $\nu >0$ 
because reasoning used in the above proof of \pref{;pduf} is invertible. 

\noindent 
\thetag{2} For every $t>0$, by adopting the identity 
\begin{align}\label{;eqadopt}
 \bigl( 
 C\eb{t}\cosh x+\hbe (A_{t}),\,Z_{t}
 \bigr) 
 \eqd \left( 
 C\cosh (x+B_{t}),\,Z_{t}
 \right) 
\end{align}
from \eqref{;eqtm1}, the identity \eqref{;eqpduf} extends to 

\begin{align}\label{;eqpdufe}
  \left( \frac{e^{2\db{-\nu }_{t}}}{2\ga _{\nu }}
  +\da{-\nu }_{t},\,e^{-\db{-\nu }_{t}}\!\!\da{-\nu }_{t}
  \right) 
  \eqd 
  \left( 
  \frac{1}{2\ga _{\nu }},\,e^{-\db{-\nu }_{t}}\!\!\da{-\nu }_{t}
  \right) , 
\end{align}
where $B$ and $\ga _{\nu }$ are independent on 
both sides. In fact, by \eqref{;eqadopt}, it holds that for any 
bounded measurable function 
$f:\R \to \R $, 
 \begin{align*}
  \ex \!\left[ f(Z_{t})
  \exp \left( 
  -z\eb{t}\cosh x-\frac{z^{2}}{2}A_{t}
  \right) 
  \right] 
  =\ex \!\left[ f(Z_{t})
  \exp \left\{ 
  -z\cosh (x+B_{t})
  \right\} 
  \right] , 
 \end{align*}
to which the same reasoning as in the proof of 
\pref{;pduf} applies. 
The extension \eqref{;eqpdufe} may also be seen as an immediate 
consequence of \cite[Proposition~13.1]{myPII} together with 
\cite[Proposition~1.7]{myPI}. In view of those two propositions, 
the above identity \eqref{;eqpdufe} remains true when we replace 
$t$ by any positive and finite stopping time of the process 
$\bigl\{ e^{-\db{-\nu }_{t}}\!\!\da{-\nu }_{t}\bigr\} _{t\ge 0}$. 

\noindent 
\thetag{3} As a consequence of one of their results relevant to 
hyperbolic Bessel processes, the equality \eqref{;jlt2} is obtained 
in Proposition~2.4 of \cite{jw} by Jakubowski and Wi\'sniewolski, 
who also give an alternative proof of Bougerol's identity \eqref{;boug1} 
based on their study. 
\end{rem}

We close this subsection with a representation for the joint 
Laplace transform of the law of $(\eb{t},\,A_{t})$ 
in terms of $B_{t}$, which is easily deduced from proofs 
of \psref{;ppar} and \ref{;pduf}, and which may be compared 
with \rref{;rtm3}\,\thetag{1}. 

\begin{prop}\label{;pjlt}
 Let $\la \ge 0$ and $\xi \in \R $. Then for every $t>0$, the expectation 
 \begin{align*}
  \ex \!\left[ 
  \exp \left( 
  -\la \eb{t}-\frac{\xi ^{2}}{2}A_{t}
  \right) 
  \right] 
 \end{align*}
 admits the following representation: 
 \begin{align}
  \ex \!\left[ 
  \exp \left( -\la \cosh B_{t}\right) 
  \cos \bigl( \sqrt{\xi ^{2}-\la ^{2}}\sinh B_{t}\bigr) 
  \right] 
  && \text{if}\ \la \le |\xi |, \label{;eqjlt1}
 \intertext{and}
  \ex \!\left[ 
  \exp \left( -\la \cosh B_{t}\right) 
  \cosh \bigl( \sqrt{\la ^{2}-\xi ^{2}}\sinh B_{t}\bigr) 
  \right] 
  && \text{if}\ \la \ge |\xi |. \label{;eqjlt2}
 \end{align}
\end{prop}

\begin{proof}
 We recall from the proof of \pref{;ppar} the equality \eqref{;jlt1}, 
 which may be restated, by symmetry of $\sinh B_{t}$, as 
 \begin{align*}
  \ex \!\left[ 
  \exp \left( 
  -|a\xi |\eb{t}-\frac{\xi ^{2}}{2}A_{t}
  \right) 
  \right] 
  =\ex \!\left[ 
  \exp \left( 
  -|a\xi |\cosh B_{t}
  \right) 
  \cos \bigl( 
  \xi \sqrt{1-a^{2}}\sinh B_{t}
  \bigr) 
  \right] . 
 \end{align*}
 Here $a\in [-1,1]$ and $\xi \in \R $ are arbitrary. Putting 
 $|a\xi |=\la \le |\xi |$ shows \eqref{;eqjlt1}. To see \eqref{;eqjlt2}, 
 we recall from the proof of \pref{;pduf} the equality \eqref{;jlt2}, 
 which asserts that for every $z\ge 0$, 
 \begin{align*}
  &\ex \!\left[ 
  \exp \left( 
  -z\eb{t}\cosh x-\frac{z^{2}}{2}A_{t}
  \right) 
  \right] \\
  &=\ex \!\left[ 
  \exp \left( 
  -z\cosh x\cosh B_{t}-z\sinh x\sinh B_{t}
  \right) 
  \right] \\
  &=\ex \!\left[ 
  \exp \left( -z\cosh x\cosh B_{t}\right) 
  \cosh \bigl( 
  z\sinh x\sinh B_{t}
  \bigr) 
  \right] , 
 \end{align*}
 where the second equality is due to symmetry of $\sinh B_{t}$. 
 Writing $z\cosh x=\la $ and $z=|\xi |$ leads to \eqref{;eqjlt2}
 and finishes the proof. 
\end{proof}

With the help of \tref{;tm1}, we may replace $t$ in the 
statement of \pref{;pjlt} by any positive and finite 
stopping time $\tau $ of $Z$; in fact, 
similarly to \pref{;pjlt}, we see from \pref{;pkey}\,\thetag{i} 
that for every $\mu >0$, 
\begin{equation}\label{;eqjlt3}
 \begin{split}
 &\ex \!\left[ 
 \exp \left( 
 -\la e^{z_{\mu }}-\frac{\xi ^{2}}{2}\frac{e^{z_{\mu }}}{\mu }
 \right) 
 \right] \\
 &=
 \begin{cases}
  \ex \!\left[ 
  \exp (-\la \cosh z_{\mu })
  \cos \bigl( 
  \sqrt{\xi ^{2}-\la ^{2}}\sinh z_{\mu }
  \bigr) 
  \right] & \text{if}\ \la \le |\xi |, \\[2mm]
  \ex \!\left[ 
  \exp (-\la \cosh z_{\mu })
  \cosh \bigl( 
  \sqrt{\la ^{2}-\xi ^{2}}\sinh z_{\mu }
  \bigr) 
  \right] & \text{if}\ \la \ge |\xi |, \\
 \end{cases}
\end{split}
\end{equation}
and if we integrate both sides with respect to 
$\pr (1/Z_{\tau }\in d\mu )$, then the above-mentioned 
replacement is verified thanks to \pref{;pparti}. 
Note that by using \eqref{;idengig} and \eqref{;k0}, 
the left-hand side of \eqref{;eqjlt3} is computed, 
regardless of whether $\la \le |\xi |$ or not, as 
\begin{align*}
 \frac{
 K_{0}\bigl( \sqrt{(\mu +\la )^{2}+\xi ^{2}-\la ^{2}}\,\bigr) 
 }{K_{0}(\mu )}, 
\end{align*}
which 
also reveals the equality 
\begin{align*}
 \ex \!\left[ 
 \exp \left( 
 -\la \eb{\tau }-\frac{\xi ^{2}}{2}A_{\tau }
 \right) 
 \right] 
 =\ex \!\left[ 
 \frac{
 K_{0}\bigl( 
 \sqrt{(1/Z_{\tau }+\la )^{2}+\xi ^{2}-\la ^{2}}\,
 \bigr) 
 }{K_{0}(1/Z_{\tau })}
 \right] . 
\end{align*}
When $\tau =t>0$, these expectations are expressed as 
\begin{align}\label{;jltrepr}
 2\int _{0}^{\infty }\frac{d\mu }{\mu }\,
 K_{0}\bigl( \sqrt{(\mu +\la )^{2}+\xi ^{2}-\la ^{2}}\,\bigr) 
 \Theta _{\mu }(t)
\end{align}
owing to \eqref{;izlaw}. We will return to the last 
expression in \ssref{;expl} of the appendix. 

\subsection{A symmetry for laws of $e^{2B_{\tau }}v+A_{\tau },\,v>0$}\label{;sym}

In this subsection, we prove 

\begin{prop}\label{;psym}
 Let $\tau $ be a stopping time of the process 
 $Z$ such that $0<\tau <\infty $ a.s. Then 
 there takes place the coincidence 
 \begin{equation}\label{;claw}
  \begin{split}
   &\frac{1}{v}\exp \left( -\frac{1}{2v}\right) 
   \pr \left( e^{2B_{\tau }}v+A_{\tau }\in du,\,
   Z_{\tau }\in dw\right) dv\\
   &=\frac{1}{u}\exp \left( -\frac{1}{2u}\right) 
   \pr \left( e^{2B_{\tau }}u+A_{\tau }\in dv,\,
   Z_{\tau }\in dw\right) du 
  \end{split}
 \end{equation}
 as measures on $(0,\infty )^{3}$. 
\end{prop}

\begin{proof}
 Observe from \eqref{;eqlkeyd} that for any 
 $\mu ,\xi >0$, we have 
 \begin{align}\label{;eqpsym1}
  \int _{\R }dy\,e^{-\mu \cosh y}e^{i\xi \sinh (x+y)}
  =\int _{\R }dy\,e^{-\xi \cosh y}e^{i\mu \sinh (x+y)}
 \end{align}
 for all $x\in \R $. Let $f:(0,\infty )\to \R $ be 
 bounded and measurable. 
 Upon substituting $x$ by $B_{\tau }$ in \eqref{;eqpsym1}, 
 we multiply both sides by $f(Z_{\tau })$ 
 and take the expectation to obtain 
 \begin{equation}\label{;eqpsym2}
  \begin{split}
  &\int _{\R }dy\,e^{-\mu \cosh y}
  \ex \!\left[ 
  f(Z_{\tau })\exp \left\{ 
  i\xi \sinh (y+B_{\tau })
  \right\} 
  \right] \\
  &=\int _{\R }dy\,e^{-\xi \cosh y}
  \ex \!\left[ 
  f(Z_{\tau })\exp \left\{ 
  i\mu \sinh (y+B_{\tau })
  \right\} 
  \right] 
 \end{split}
 \end{equation}
 by Fubini's theorem. Recalling from \tref{;tm1}, as well as from 
 \tref{;tm2}, the identity 
 \begin{align*}
  \left( 
  \sinh (y+B_{\tau }),\, 
  Z_{\tau }
  \right) 
  \eqd \left( 
  \eb{\tau}\!\sinh y+\beta (A_{\tau }), \, 
  Z_{\tau }
  \right) 
 \end{align*}
 for every $y\in \R $, we rewrite the left-hand side of 
 \eqref{;eqpsym2} as 
 \begin{align*}
  &\int _{\R }dy\,e^{-\mu \cosh y}
  \ex \!\left[ 
  f(Z_{\tau })
  e^{i\xi \eb{\tau }\!\sinh y}
  \exp \left( 
  -\frac{\xi ^{2}}{2}A_{\tau }
  \right) 
  \right] \\
  &=\int _{0}^{\infty }\frac{dv}{v}\,
  \exp \left( -\frac{1}{2v}\right) 
  \exp \left( 
  -\frac{\mu ^{2}}{2}v
  \right) 
  \ex \!\left[ 
  f(Z_{\tau })
  e^{i\xi \eb{\tau }\beta (v)}
  \exp \left( 
  -\frac{\xi ^{2}}{2}A_{\tau }
  \right) 
  \right] \\
  &=\int _{0}^{\infty }\frac{dv}{v}\,
  \exp \left( -\frac{1}{2v}\right) 
  \exp \left( 
  -\frac{\mu ^{2}}{2}v
  \right) \ex \!\left[ 
  f(Z_{\tau })\exp \left\{ 
  -\frac{\xi ^{2}}{2}\left( 
  e^{2B_{\tau }}v+A_{\tau }
  \right) 
  \right\} 
  \right] , 
\end{align*}
where, $\beta $ being a Brownian motion independent of $B$, 
we used \pref{;phinge} for the second line. 
Since the last integral agrees, by \eqref{;eqpsym2}, 
with the one in which $\mu $ and $\xi $ are 
exchanged, we obtain the statement of the proposition 
thanks to arbitrariness of $f$, $\mu $ and $\xi $. 
\end{proof}

\begin{rem}\label{;rfinite}
When $\tau $ is only assumed to be finite a.s., the relation 
\eqref{;claw} is valid by restricting probabilities on 
both sides to the event $\{ \tau >0\} $, which is because 
in this case, the third line of the last displayed 
equation in the above proof is further rewritten as 
\begin{align*}
 &2K_{0}\bigl( \sqrt{\mu ^{2}+\xi ^{2}}\bigr) 
 f(0)\pr (\tau =0)\\
 &+
 \int _{0}^{\infty }\frac{dv}{v}\,
  \exp \left( -\frac{1}{2v}\right) 
  \exp \left( 
  -\frac{\mu ^{2}}{2}v
  \right) 
  \ex \!\left[ 
  f(Z_{\tau })\exp \left\{ 
  -\frac{\xi ^{2}}{2}\left( 
  e^{2B_{\tau }}v+A_{\tau }
  \right) 
  \right\} ;\,\tau >0
  \right] 
\end{align*}
by \eqref{;k0} if $f$ is a bounded measurable function on $[0,\infty )$. 
\end{rem}

In view of \psref{;pparti}, the above \pref{;psym} may 
also be seen as a consequence of the relation 
\begin{equation}\label{;idenlaw0}
 \begin{split}
 &\frac{1}{v}\exp \left( 
 -\frac{1}{2v}
 \right) 
 \frac{\pr \left( e^{2z_{\mu}}v+e^{z_{\mu }}/\mu \in du\right) }
 {du}\\
 &=\frac{1}{u}\exp \left( 
 -\frac{1}{2u}
 \right) 
 \frac{\pr \left( e^{2z_{\mu}}u+e^{z_{\mu }}/\mu \in dv\right) }
 {dv},\quad u,v>0, 
 \end{split}
\end{equation}
for every $\mu >0$, which follows readily from \pref{;pkey} 
by the same reasoning as in the above proof. 
A simple computation shows that the left-hand side, and hence 
the right-hand side as well, of \eqref{;idenlaw0} admits the 
representation 
\begin{align*}
 \frac{1}{4K_{0}(\mu )}\frac{1}{uv}
 \biggl( 
 1+\frac{1}{\sqrt{1+4\mu ^{2}uv}}
 \biggr) \exp \left\{ 
 -\frac{u+v}{4uv}
 \left( \sqrt{1+4\mu ^{2}uv}+1\right) 
 \right\} , 
\end{align*}
which is indeed symmetric with respect to $u$ and $v$. 
From this representation together with \eqref{;izlaw}, 
it follows that when $\tau =t>0$, both sides of \eqref{;claw} 
admit the density function 
\begin{align*}
 \frac{1}{2uvw}
 \biggl( 
 1+\frac{1}{\sqrt{1+4uv/w^{2}}}
 \biggr) \exp \left\{ 
 -\frac{u+v}{4uv}
 \left( \sqrt{1+4uv/w^{2}}+1\right) 
 \right\} \Theta _{1/w}(t)
\end{align*}
with respect to the Lebesgue measure $dudvdw$ on $(0,\infty )^{3}$. 

\begin{rem}\label{;rlt}
 For every $v>0$, the Laplace transform of the law of 
 $e^{2B_{\tau }}v+A_{\tau }$ may be expressed as follows: 
 for any $\xi \in \R $, 
 \begin{align*}
  &\ex \!\left[ 
  \exp \left\{ 
  -\frac{\xi ^{2}}{2}\left( 
  e^{2B_{\tau }}v+A_{\tau }
  \right) 
  \right\} 
  \right] \\
  &=\ex \!\left[ 
  \cos \left\{ 
  \xi \sinh \bigl( 
  \argsh \hbe (v)+B_{\tau }
  \bigr) 
  \right\} 
  \right] \\
  &=\int _{\R }\frac{dy}{\sqrt{2\pi v}}\cosh y\exp \left( 
  -\frac{\sinh ^{2}y}{2v}
  \right) 
  \ex \!\left[ 
  \cos \left\{ \xi \sinh (y+B_{\tau })\right\} 
  \right] , 
 \end{align*}
 where in the second line, $\hbe $ is a Brownian motion 
 independent of $B$, and the third line follows by 
 a simple computation. To see the first equality, it 
 suffices to note that by \eqref{;eqtm1}, 
 \begin{align*}
  \sinh \bigl( 
  \argsh \hbe (v)+B_{\tau }
  \bigr) 
  &\eqd 
  \eb{\tau }\hbe (v)+\beta (A_{\tau })\\
  &\eqd \hbe \left( 
  e^{2B_{\tau }}v+A_{\tau }
  \right) , 
 \end{align*}
 where in the second line, we still denote by $\hbe $ a Brownian motion 
 independent of $B$.  
\end{rem}

\subsection{Analytic applications of \pref{;phinge} and \lref{;lkey}}

In this subsection, we give one analytic application of 
\pref{;phinge}, as well as other two analytic applications of 
\lref{;lkey}, one of which uses the case $\xi =0$, and 
the other of which uses the case $\la =0$ as has already 
been applied in \ssref{;prftm3}. 

We start with derivation of the following integral representation 
for the Macdonald function $K_{\nu }$ with $\nu >-1/2$, 
by means of \pref{;phinge}. 

\begin{prop}\label{;pintk3}
 If $\nu >-1/2$, then it holds that 
 \begin{align}\label{;intk3}
  K_{\nu }(z)
  =\frac{\sqrt{\pi }z^{\nu }}{2^{\nu }\Gamma (\nu +1/2)}
  \int _{0}^{\infty }dx\,
  e^{-z\cosh x}\sinh ^{2\nu }\!x,\quad z>0. 
 \end{align}
\end{prop}

For the representation \eqref{;intk3}, see, e.g., 
\cite[p.~140, Problem~6]{leb}. 

\begin{proof}[Proof of \pref{;pintk3}]
 By symmetry and by \pref{;phinge}, the integral in the 
 right-hand side of \eqref{;intk3} is equal to 
 \begin{align}\label{;eqpintk3}
  &\frac{1}{2}\int _{\R }dx\,e^{-z\cosh x}|\sinh x|^{2\nu } \notag \\
  &=\frac{1}{2}\int _{0}^{\infty }
  \frac{dv}{v}\,\exp \left( 
  -\frac{1}{2v}
  \right) \exp \left( 
  -\frac{z^{2}}{2}v
  \right) \ex \!\bigl[ 
  \left| 
  \beta (v)
  \right| ^{2\nu }
  \bigr] . 
 \end{align}
 Note that when $\nu >-1/2$, we have 
 \begin{align*}
  \ex \!\bigl[ 
  \left| 
  \beta (v)
  \right| ^{2\nu }
  \bigr] 
  &=v^{\nu }\ex \!\bigl[ 
  \left| 
  \beta (1)
  \right| ^{2\nu }
  \bigr] \\
  &=v^{\nu }\times \frac{2^{\nu }}{\sqrt{\pi }}\Gamma (\nu +1/2)
 \end{align*}
 for every $v>0$, where the first line is due to the scaling property of 
 Brownian motion and the second follows readily from the fact 
 that $|\beta (1)|^{2}\eqd 2\ga _{1/2}$. By plugging the last expression 
 into \eqref{;eqpintk3}, we see that the right-hand side of the claimed 
 formula \eqref{;intk3} is written as 
 \begin{align*}
  \frac{1}{2}z^{\nu }\int _{0}^{\infty }dv\,
  v^{\nu -1}\exp \left\{ 
  -\frac{1}{2}\left( 
  \frac{1}{v}+z^{2}v
  \right) 
  \right\} , 
 \end{align*}
 which is equal to $K_{\nu }(z)$ in view of \eqref{;intk2} as well as 
 of \eqref{;gig}. 
\end{proof}

We turn to the two applications of \lref{;lkey}. The first one 
concerns an integral representation of products of two Macdonald 
functions, which is found, e.g., in \cite[p.~140, Problem~7]{leb}. 

\begin{prop}\label{;pint}
 For every $\nu \in \R $, it hold that 
 \begin{align}\label{;intrepr}
  K_{\nu }(z)K_{\nu }(w)=\frac{1}{2}\int _{0}^{\infty }
  \frac{dv}{v}\,\exp \left( 
  -\frac{1}{2v}
  \right) \exp \left( 
  -\frac{z^{2}+w^{2}}{2}v
  \right) K_{\nu }(zwv), \quad z,w>0. 
\end{align}
\end{prop}

\begin{proof}
 By the integral representation \eqref{;intk1} of $K_{\nu }$, 
 we plug the expression 
 \begin{align*}
  K_{\nu }(zwv)=\frac{1}{2}\int _{\R }dx\,e^{-zwv\cosh x}e^{-\nu x}
 \end{align*}
 into the right-hand side of \eqref{;intrepr}. Then by Fubini's theorem, 
 it is rewritten as 
 \begin{align}\label{;eqpint}
  \frac{1}{4}\int _{\R }dx\,e^{-\nu x}\int _{0}^{\infty }
  \frac{dv}{v}\,\exp \left( 
  -\frac{1}{2v}
  \right) \exp \left( 
  -\frac{z^{2}+w^{2}}{2}v
  \right) e^{-zwv\cosh x}. 
 \end{align}
 Applying \lref{;lkey} with $\mu =z$, $\la =w$ and $\xi =0$, we 
 see that the integral with respect to $v$ above is equal to 
 \begin{align*}
  \int _{\R }dy\,e^{-z\cosh y}e^{-w\cosh (x+y)}. 
 \end{align*}
 Therefore by using Fubini's theorem again, the expression 
 \eqref{;eqpint} is further rewritten as 
 \begin{align*}
  \frac{1}{2}\int _{\R }dy\,e^{-z\cosh y}e^{\nu y}\times 
  \frac{1}{2}\int _{\R }dx\,e^{-w\cosh (x+y)}e^{-\nu (x+y)}, 
 \end{align*}
 which agrees with the left-hand side of \eqref{;intrepr} 
 in view of \eqref{;intk1}. 
\end{proof}

The second application of \lref{;lkey} deals with an integral representation 
for the density function of a given symmetric random variable. 

\begin{prop}\label{;pden}
 Let $X$ be a symmetric random variable and suppose that it 
 satisfies 
 \begin{align}\label{;cond}
  \int _{\R }dx\,e^{-\mu \cosh x}\!\int _{\R }d\xi\,
  \bigl| 
  \ex \!\left[ 
  \cos \left\{ 
  \xi \sinh (x+X)
  \right\} 
  \right] 
  \bigr| <\infty  \quad \text{for any $\mu >0$.}
 \end{align}
 Then $X$ admits the density function $\vp $ given by 
 \begin{align}\label{;denrepr1}
  \vp (x)=\frac{1}{2\pi }\int _{\R }d\xi\,
  \ex \!\left[ 
  \cos \left\{ 
  \xi \sinh (x+X)
  \right\} 
  \right] ,\quad x\in \R . 
\end{align} 
\end{prop}

\begin{proof}
 What we are going to use is \lref{;lkey} with $\la =0$, namely 
 the integral identity \eqref{;eqlkeyd}. We substitute $x$ by $X$ 
 and take the expectation on both sides to get 
 \begin{equation}\label{;eqpden1}
  \begin{split}
  &\int _{\R }dy\,e^{-\mu \cosh y}
  \ex \!\left[ 
  \cos \left\{ 
  \xi \sinh (y+X)
  \right\} 
  \right] \\
  &=\ex \!\left[ 
  \int _{0}^{\infty }\frac{dv}{v}\,\exp \left( 
  -\frac{1}{2v}
  \right) \exp \left( 
  -\frac{\mu ^{2}+\xi ^{2}}{2}v
  \right) e^{i\mu \xi v\sinh X}
  \right] , 
  \end{split}
 \end{equation}
 where on the left-hand side, we used Fubini's theorem and 
 the fact that the mapping 
 $
 \R \ni y\mapsto \ex \!\left[ 
 \sin \left\{ \xi \sinh (y+X)\right\} \right] 
 $
 is an odd function due to symmetry of $X$. 
 We integrate both sides of the above equality with respect to 
 $\xi $ over $\R $. Then by the condition \eqref{;cond} and 
 by Fubini's theorem, the left-hand side of \eqref{;eqpden1} 
 turns into  
 \begin{align}\label{;eqpden2}
  2\pi \int _{\R }dy\,e^{-\mu \cosh y}\vp (y)
 \end{align}
 with function $\vp $ given in \eqref{;denrepr1}. On the other 
 hand, as for the right-hand side of \eqref{;eqpden1}, by 
 observing that 
 \begin{align*}
  \int _{\R }d\xi \,\exp \left( 
  -\frac{\xi ^{2}}{2}v
  \right) e^{i\mu \xi v\sinh X}
  =\sqrt{\frac{2\pi }{v}}\exp \left( 
  -\frac{\mu ^{2}\sinh ^{2}X}{2}v
  \right) , 
 \end{align*}
 Fubini's theorem allows us to compute 
 \begin{align*}
  &2\pi \ex \!\left[ 
  \int _{0}^{\infty }\frac{dv}{\sqrt{2\pi v^{3}}}\,
  \exp \left( 
  -\frac{1}{2v}
  \right) 
  \exp \left( 
  -\frac{\mu ^{2}\cosh ^{2}X}{2}v
  \right) 
  \right] \\
  &=2\pi \ex \!\left[ 
  \exp \left\{ 
  -\frac{\mu ^{2}\cosh ^{2}X}{2}\tau _{1}(\hB )
  \right\} 
  \right] \\
  &=2\pi \ex \!\left[ 
  \exp \left( -\mu \cosh X\right) 
  \right] 
 \end{align*}
 thanks to \eqref{;taulaw} and \eqref{;lt}, where in the second 
 line, $\hB $ denotes a Brownian motion independent of $X$. 
 Since the last expression 
 agrees with \eqref{;eqpden2} for any $\mu >0$, we obtain the 
 conclusion owing to symmetry of $X$. 
\end{proof}

We may compare the above proposition with the well-known fact 
(see, e.g., \cite[Theorem~3.3.5]{dur}) that if a generic random variable 
$X$ satisfies the condition 
$\int _{\R }d\xi \,\bigl| \ex [e^{i\xi X}]\bigr| <\infty $, then it 
admits the density function 
\begin{align}\label{;denrepr2}
 \vp (x)
 =\frac{1}{2\pi }\int _{\R }d\xi \,e^{-ix\xi }\ex \!\left[ e^{i\xi X}\right] ,
 \quad x\in \R . 
\end{align}
If we apply this general fact to a symmetric $X$, then the integral 
representation \eqref{;denrepr2} may be written as 
\begin{align}\label{;denrepr2d}
 \vp (x)&=\frac{1}{2\pi }\int _{\R }d\xi \,\cos (x\xi )
 \ex \!\left[ \cos (\xi X)\right] \notag \\
 &=\frac{1}{2\pi }\int _{\R }d\xi \,\ex \!\left[ 
 \cos \{ \xi (x+X)\}  
 \right] ,\quad x\in \R . 
\end{align}
\pref{;pden} asserts that for symmetric random variables $X$ 
of a certain class, two expressions \eqref{;denrepr1} and 
\eqref{;denrepr2d} agree. 

We give examples of symmetric random variables satisfying the 
condition \eqref{;cond}. 

\begin{exm}\label{;esymX}
\thetag{1} Brownian motion at fixed time $t>0$ fulfills 
\eqref{;cond}. To verify it, note that by the 
identity \eqref{;boug2}, 
\begin{align*}
 \bigl| 
 \ex \!\left[ \cos \left\{ \xi \sinh (x+B_{t})\right\} 
 \right] 
 \bigr| 
 &=\left| 
 \ex \!\left[ 
 \cos \left( \xi \eb{t}\sinh x\right) 
 \exp \left( 
 -\frac{\xi ^{2}}{2}A_{t}
 \right) 
 \right] 
 \right| \\
 &\le 
 \ex \!\left[ 
 \exp \left( 
 -\frac{\xi ^{2}}{2}A_{t}
 \right) 
 \right] 
\end{align*}
for any $x,\xi \in \R $. Therefore the double integral 
in the condition \eqref{;cond} is dominated by 
\begin{align*}
 &\int _{\R }dx\,e^{-\mu \cosh x}\times \int _{\R }d\xi \,
 \ex \!\left[ 
 \exp \left( 
 -\frac{\xi ^{2}}{2}A_{t}
 \right) 
 \right] \\
 &=2\sqrt{2\pi }K_{0}(\mu )\ex \!\left[ 
 \frac{1}{\sqrt{A_{t}}}
 \right] , 
\end{align*}
which is finite by the fact that 
\begin{align*}
 \ex \!\left[ 
 \frac{1}{\sqrt{A_{t}}}
 \right] =\frac{1}{\sqrt{t}}
\end{align*}
for any $t>0$. The above fact may be seen from the 
identity \eqref{;boug2} in such a way that two density 
functions in $y\in \R $: 
\begin{align}
 \frac{\pr \left( 
 \eb{t}\sinh x+\beta (A_{t})\in dy
 \right) }
 {dy}
 &=\ex \!\left[ 
 \frac{1}{\sqrt{2\pi A_{t}}}
 \exp \left\{ 
 -\frac{\left( 
 y-\eb{t}\sinh x
 \right) ^{2}}{2A_{t}}
 \right\} 
 \right] \label{;eden1}
\intertext{and }
 \frac{
 \pr \left( \sinh (x+B_{t})\in dy\right) 
 }
 {dy}
 &=\frac{1}{\sqrt{2\pi t}}\frac{1}{\sqrt{1+y^{2}}}
 \exp \left\{  
 -\frac{(\argsh y-x)^{2}}{2t}
 \right\} \label{;eden2}
\end{align}
agree for every $x\in \R $ and evaluating them at 
$x=y=0$ leads to the claimed equality. 
It should also be noted that since $X=B_{t}$ fulfills 
$\int _{\R }d\xi \,\bigl| \ex [e^{i\xi X}]\bigr| <\infty $, 
the expression \eqref{;denrepr2d} is valid as well; 
in fact, a direct computation shows that 
\begin{align*}
 \frac{1}{2\pi }\int _{\R }d\xi \,\cos (x\xi )
 \exp \left( 
 -\frac{t}{2}\xi ^{2}
 \right) 
 =\frac{1}{\sqrt{2\pi t}}\exp \left( 
 -\frac{x^{2}}{2t}
 \right) ,\quad x\in \R . 
\end{align*}

\noindent 
\thetag{2} From the above argument in \thetag{1}, it is now clear 
that Brownian motion evaluated at an independent random time 
$T$ satisfying 
\begin{align*}
 \ex \!\left[ 
 \frac{1}{\sqrt{T}}
 \right] <\infty , 
\end{align*}
fulfills \eqref{;cond}. A typical example of such 
situations is given by a symmetric Cauchy variable 
$aC$ for every $a\neq 0$, 
because it holds that  
$aC\eqd \beta (\tau _{a}(\hB ))$ 
as seen in \eqref{;cauchy} and that by \eqref{;taulaw}, 
$\tau _{a}(\hB )$ satisfies 
\begin{align*}
 \ex \!\left[ 
 \frac{1}{\sqrt{\mathstrut \tau _{a}(\hB )}}
 \right] 
 &=\frac{|a|}{\sqrt{2\pi }}\int _{0}^{\infty }\frac{dv}{v^{2}}\,
 \exp \left( 
 -\frac{a^{2}}{2v}
 \right) \\
 &<\infty . 
\end{align*}
Among other examples,  we have 
$(1/Z_{t})\sinh B_{t}$ and $\bigl( 1/\sqrt{Z_{t}}\bigr) \sinh B_{t}$ 
for each $t>0$, as is deduced from identities \eqref{;eqrstop2} 
and \eqref{;eqrstop3}. 

\noindent 
\thetag{3} For each $u>0$, the random variable $z_{u}$ fulfills 
\eqref{;cond}. (Here we replace by $u$ the exponent 
$\mu $ in the definition \eqref{;zlaw} of $z_{\mu }$ to avoid confusion.) 
More generally, if a symmetric random variable $X$ is such that 
\begin{align*}
 \pr (X\in dx)
 =\int _{0}^{\infty }m(du)\,\pr (z_{u}\in dx),\quad x\in \R , 
\end{align*}
for some probability measure $m$ on $(0,\infty )$ satisfying 
\begin{align}\label{;condd}
 \int _{0}^{\infty }m(du)\,\sqrt{u}\ex \!\left[ 
 e^{-z_{u}/2}
 \right] <\infty , 
\end{align}
then $X$ fulfills \eqref{;cond}. Examples in \thetag{1} and 
\thetag{2} above may be seen as consequences of this general 
statement; indeed, if we take $m(du)=\pr (1/Z_{t}\in du)$ 
for $t>0$, then by \pref{;pparti}, 
\begin{align*}
 \int _{0}^{\infty }m(du)\,
 \sqrt{u}\ex \!\left[ 
 e^{-z_{u}/2}
 \right] 
 =\ex \!\left[ 
 \frac{1}{\sqrt{A_{t}}}
 \right] . 
\end{align*}
In order to draw the above condition \eqref{;condd}, we use 
in place of \eqref{;boug2} the identity 
\begin{align*}
 \sinh (x+z_{u})\eqd 
 e^{z_{u}}\sinh x+\beta (e^{z_{u}}/u)
\end{align*}
adopted from \pref{;pkey} and 
argue along the same lines as in \thetag{1} to see that 
\begin{align*}
 \int _{\R }d\xi \,\bigl| 
 \ex \!\left[ 
 \cos \left\{ 
 \xi \sinh (x+z_{u})
 \right\} 
 \right] 
 \bigr| 
 &\le \int _{\R }d\xi \,\ex \!\left[ 
 \exp \left( 
 -\frac{\xi ^{2}}{2}\frac{e^{z_{u}}}{u}
 \right) 
 \right] \\
 &=\sqrt{2\pi u}\ex \!\left[ 
 e^{-z_{u}/2}
 \right] , 
\end{align*}
which leads to \eqref{;condd}. 
If we note the identities 
\begin{align*}
 \ex \!\left[ 
 e^{-z_{u}/2}
 \right] 
 &=\frac{K_{1/2}(u)}{K_{0}(u)}\\
 &=\sqrt{
 \frac{\pi }{2u}
 }\frac{e^{-u}}{K_{0}(u)}, 
\end{align*}
then the condition \eqref{;condd} is restated as 
$\int _{0}^{\infty }m(du)\,e^{-u}/K_{0}(u)<\infty $. In the last 
displayed equations, we used \eqref{;intk1} for the first line 
and the explicit expression of $K_{1/2}$ 
(see, e.g., \cite[Equation~(5.8.5)]{leb}) for the second. 
\end{exm}

As already seen partly in \thetag{1} of the above example, 
in the case $X=B_{t}$ for $t>0$, the right-hand side of 
\eqref{;denrepr1} is expressed, by the identity 
\eqref{;boug2}, as 
\begin{align*}
 \frac{1}{2\pi }\int _{\R }d\xi \,\ex \!\left[ 
 \cos \left( \xi \eb{t}\sinh x\right) 
 \exp \left( 
 -\frac{\xi ^{2}}{2}A_{t}
 \right) 
 \right] . 
\end{align*}
By Fubini's theorem, this expression is 
equal to 
\begin{align*}
  \frac{1}{2\pi }\ex \!\left[ 
  \int _{\R }d\xi \,\cos \left( \xi \eb{t}\sinh x\right) 
  \exp \left( 
  -\frac{\xi ^{2}}{2}A_{t}
  \right) 
  \right] 
  =\ex \!\left[ 
  \frac{1}{\sqrt{2\pi A_{t}}}\exp 
  \left( 
  -\frac{e^{2B_{t}}\sinh ^{2}x}{2A_{t}}
  \right) 
  \right] , 
\end{align*}
and hence \pref{;pden} entails the relation 
\begin{align}\label{;grel}
 \ex \!\left[ 
  \frac{1}{\sqrt{2\pi A_{t}}}\exp 
  \left( 
  -\frac{e^{2B_{t}}\sinh ^{2}x}{2A_{t}}
  \right) 
  \right] 
  =\frac{1}{\sqrt{2\pi t}}\exp \left( 
  -\frac{x^{2}}{2t}
  \right) 
\end{align}
for every $t>0$ and $x\in \R $. The above probabilistic representation 
for the Gaussian kernel is also obtained by evaluating 
\eqref{;eden1} and \eqref{;eden2} at $y=0$. Moreover, thanks to 
the formula \eqref{;jl}, the left-hand side of \eqref{;grel} is 
calculated as 
\begin{align*}
 &\int _{0}^{\infty }\frac{dv}{v}\int _{0}^{\infty }
 \frac{du}{u}\,\exp \left( 
 -\frac{1+u^{2}}{2v}
 \right) \Theta _{u/v}(t)
 \frac{1}{\sqrt{2\pi v}}\exp \left( 
 -\frac{u^{2}\sinh ^{2}x}{2v}
 \right) \\
 &=\int _{0}^{\infty }\frac{du}{u}\,\Theta _{u}(t)
 \int _{0}^{\infty }\frac{dv}{\sqrt{2\pi v^{3}}}
 \exp \left( 
 -\frac{1}{2v}
 \right) \exp \left( 
 -\frac{u^{2}\cosh ^{2}x}{2}v
 \right) \\
 &=\int _{0}^{\infty }\frac{du}{u}\,\Theta _{u}(t)e^{-u\cosh x}, 
\end{align*}
where for the second line, we changed the variable 
$u$ into $vu$ and used Fubini's theorem, 
and for the third line, we used \eqref{;taulaw} and 
\eqref{;lt}. Therefore in addition to \eqref{;ltt}, we have 
another characterization of 
the function $\Theta _{r}(t),\,r>0,t>0$, in terms of the 
Laplace transform in variable $r$: 
\begin{align}\label{;ltr}
 \int _{0}^{\infty }\frac{dr}{r}\,\Theta _{r}(t)e^{-r\cosh x}
 =\frac{1}{\sqrt{2\pi t}}\exp \left( 
  -\frac{x^{2}}{2t}
  \right) ,\quad t>0,\,x\in \R .
\end{align}
We remark that this relation is stated in 
\cite[Proposition~4.5\,(i)]{myPI} and, 
as observed in \cite[Proposition~4.2]{mySI}, 
also follows by simply integrating both sides of 
\eqref{;jl} with respect to $v$ over $(0,\infty )$. 
As seen above, the relation \eqref{;ltr} explains the 
coincidence of the two expressions 
\eqref{;eden1} and \eqref{;eden2} in the case $y=0$. 
It is not hard to see  similarly that their coincidence in the case 
$y\neq 0$ is also reduced to the above relation; in fact, 
by using \eqref{;jl}, 
the expectation in \eqref{;eden1} is calculated as 
\begin{align*}
 \frac{1}{\sqrt{1+y^{2}}}
 \int _{0}^{\infty }\frac{du}{u}\,\Theta _{u}(t)
 \exp \left\{ 
 -u\cosh \left( \argsh y-x\right) 
 \right\} , 
\end{align*}
which agrees with \eqref{;eden2} thanks to the relation \eqref{;ltr}. 
In \ssref{;dhw} of the appendix, we derive from \eqref{;ltr} 
the integral representation \eqref{;hw} for $\Theta _{r}(t)$. 

\section{Concluding Remarks}\label{;cr}

In this paper we have shown, with the help of \pref{;pparti} 
due to Matsumoto and Yor, that Bougerol's identity 
\eqref{;boug1} as well as its extensions in \tsref{;tm1} and 
\ref{;tm2} are obtained from relevant properties of  
random variables $z_{\mu },\,\mu >0$, defined in 
\eqref{;zlaw}. In particular, we have shown that for every 
fixed $t>0$ and $x\in \R $, there takes place the coincidence 
of joint laws 
\begin{align}\label{;cjl}
 \left( 
  \eb{t}\sinh x+\beta (A_{t}), \, 
  e^{-B_{t}}A_{t}
  \right) 
  \eqd 
  \left( 
  \sinh (x+B_{t}),\,e^{-B_{t}}A_{t}
  \right) .  
\end{align}
In view of an expression \eqref{;anothery} of the process 
$Y^{x}$ defined by \eqref{;defY}, the left-hand side of 
\eqref{;cjl} is identical in law with 
\begin{align*}
 \left( Y^{x}_{t},\,e^{-B_{t}}A_{t}\right) . 
\end{align*}
Indeed, we have 
\begin{align*}
 \left( Y^{x}_{t},\,e^{-B_{t}}A_{t}\right) 
 &\eqd 
 \left( 
 e^{B_{t}}\sinh x+e^{B_{t}}\beta (e^{-2B_{t}}A_{t}),\,
 e^{-B_{t}}A_{t}
 \right) \\
 &\eqd 
 \left( \eb{t}\sinh x+\beta (A_{t}),\,e^{-B_{t}}A_{t}\right) , 
\end{align*}
where the first line follows from the relation \eqref{;2-dim} and 
the second is due to the scaling property of Brownian motion 
as was seen in \eqref{;ideny}. Recall that the process $\beta ^{x}$ 
defined in \eqref{;defbex} is a Brownian motion. 
Then, since the right-hand side of \eqref{;cjl} 
is identical in law with 
\begin{align*}
 \left( Y^{x}_{t},\,
 e^{-\beta ^{x}_{t}}\!\int _{0}^{t}e^{2\beta ^{x}_{s}}ds\right) 
\end{align*}
in view of the other expression \eqref{;expliy} of $Y^{x}$, 
we may rephrase \eqref{;cjl} as 
\begin{align}\label{;cjld}
 \left( Y^{x}_{t},\,e^{-B_{t}}\!\int _{0}^{t}e^{2B_{s}}ds\right) 
 \eqd 
 \left( 
 Y^{x}_{t},\,e^{-\beta ^{x}_{t}}\!\int _{0}^{t}e^{2\beta ^{x}_{s}}ds
 \right) . 
\end{align}
It would be interesting to give a direct explanation to the 
identity \eqref{;cjld} by means of It\^o's formula and SDEs, 
which we think should lead us to a deeper understanding, 
such as another proof, of the explicit formula \eqref{;jl} for the 
joint law of $\eb{t}$ and $A_{t}$. 

\appendix 
\section*{Appendix}
\renewcommand{\thesection}{A}
\setcounter{equation}{0}
\setcounter{prop}{0}
\setcounter{lem}{0}
\setcounter{rem}{0}

We append some explorations as to the bivariate function 
$\Theta _{r}(t),\,r>0,t>0$, that is characterized by \eqref{;ltt} 
as well as by \eqref{;ltr}, and admits the integral representation 
\eqref{;hw}. 

\subsection{An integral equation for $\Theta _{r}(t)$}\label{;ie}

In the first part of the appendix, we derive an integral equation 
satisfied by $\Theta _{r}(t)$. To begin with, we note that the 
integral representation \eqref{;intk1} for the Macdonald function 
$K_{\nu }$ is valid if $\nu $ is in $\C$, the complex plane, and that 
when $\nu =i\xi ,\,\xi \in \R $, it reads 
\begin{align}\label{;intk1i}
 K_{i\xi  }(z)
 =\frac{1}{2}\int _{\R }dx\,e^{-z\cosh x}\cos (\xi x),\quad 
 z>0 . 
\end{align}

\begin{prop}\label{;pie}
 For every $r>0$ and $t>0$, it holds that 
 \begin{align}\label{;epie}
  \Theta _{r}(t)=\frac{r}{t}\exp \left( 
  \frac{\pi ^{2}}{2t}
  \right) 
  \int _{0}^{\infty }
  \frac{du}{u(u+r)}\,K_{\pi i/t}(u+r)\Theta _{u}(t). 
 \end{align}
\end{prop}

\begin{proof}
 In the integral representation \eqref{;hw} for $\Theta _{r}(t)$, 
 the integrand is a symmetric function in $y\in \R $, and hence 
 we may represent $\Theta _{r}(t)$ as 
 \begin{align}\label{;exphw}
  \Theta _{r}(t)=
  \frac{r}{2\pi }\exp \left( 
  \frac{\pi ^{2}}{2t}
  \right) 
  \ex \!\left[ 
  \exp \left( -r\cosh B_{t}\right) \sinh B_{t}
  \sin \left( \frac{\pi }{t}B_{t}\right) 
  \right] . 
 \end{align}
 By noting the relation \eqref{;ltr} and using Fubini's theorem, 
 the expectation on the right-hand side may be written as 
 \begin{align*}
  \int _{0}^{\infty }\frac{du}{u}\,\Theta _{u}(t)
  \int _{\R }dx\,\exp \left\{ 
  -(u+r)\cosh x
  \right\} \sinh x\sin \left( \frac{\pi }{t}x\right) . 
 \end{align*}
 By applying the integration by parts formula, the integral 
 with respect to $x$ in the last expression is calculated as 
 \begin{align*}
  &\left[ 
  -\frac{1}{u+r}\exp \left\{ 
  -(u+r)\cosh x
  \right\} \sin \left( \frac{\pi }{t}x\right) 
  \right] _{x=-\infty }^{\infty }\\
  &\quad +\frac{\pi }{t(u+r)}
  \int _{\R }dx\,\exp \left\{ 
  -(u+r)\cosh x
  \right\} 
  \cos \left( \frac{\pi }{t}x\right) \\
  &=\frac{2\pi }{t(u+r)}K_{\pi i/t}(u+r)
 \end{align*}
 thanks to \eqref{;intk1i}. Hence the expectation on the right-hand side 
 of \eqref{;exphw} is equal to 
 \begin{align*}
  \frac{2\pi }{t}\int _{0}^{\infty }\frac{du}{u(u+r)}\,
  \Theta _{u}(t)K_{\pi i/t}(u+r), 
 \end{align*}
 which proves the relation \eqref{;epie} as desired. 
\end{proof}

\subsection{Derivation of \pref{;pjlt} from \eqref{;ltr}}\label{;expl}

Recall from \pref{;pjlt} the following identities between 
expectations relative to Brownian motion $B$: 
for every $\la \ge 0$ and $\xi \in \R $, 
\begin{equation}\label{;ajlt}
 \begin{split}
 &\ex \!\left[ 
 \exp \left( 
 -\la \eb{t}-\frac{\xi ^{2}}{2}A_{t}
 \right) 
 \right] \\
 &=
 \begin{cases}
  \ex \!\left[ 
  \exp \left( -\la \cosh B_{t}\right) 
  \cos \bigl( \sqrt{\xi ^{2}-\la ^{2}}\sinh B_{t}\bigr) 
  \right] & \text{if $\la \le |\xi |$,} \\[2mm]
  \ex \!\left[ 
  \exp \left( -\la \cosh B_{t}\right) 
  \cosh \bigl( \sqrt{\la ^{2}-\xi ^{2}}\sinh B_{t}\bigr) 
  \right] & \text{if $\la \ge |\xi |$.} 
 \end{cases}
 \end{split}
\end{equation}
As was seen in \eqref{;jltrepr}, the left-hand side admits the 
representation 
\begin{align}\label{;jltreprd}
 2\int _{0}^{\infty }\frac{dr}{r}\,
 K_{0}\bigl( \sqrt{(r+\la )^{2}+\xi ^{2}-\la ^{2}}\,\bigr) 
 \Theta _{r}(t), 
\end{align}
which was obtained by calculating the expectation 
\begin{align*}
 \ex \!\left[ 
 \exp \left( 
 -\la e^{z_{\mu }}-\frac{\xi ^{2}}{2}\frac{e^{z_{\mu }}}{\mu }
 \right) 
 \right] 
\end{align*}
for each $\mu >0$. In this part of the appendix, 
based on the relation \eqref{;ltr} and \pref{;phinge}, 
we derive the representation \eqref{;jltreprd} from expressions 
on the right-hand side of \eqref{;ajlt}, which we think would provide 
us with a better understanding of the identities \eqref{;ajlt}. 

We treat the case $\la \le |\xi |$ first. Note that by \eqref{;ltr} 
and Fubini's theorem, the right-hand side of \eqref{;ajlt} in this 
case is rewritten as 
\begin{align}\label{;arewr}
 \int _{0}^{\infty }\frac{dr}{r}\,\Theta _{r}(t)
 \int _{\R }dx\,\exp \left\{ 
 -(r+\la )\cosh x
 \right\} \cos \bigl( \sqrt{\xi ^{2}-\la ^{2}}\sinh x\bigr) . 
\end{align}
By \pref{;phinge}, the integral with respect to $x$ in 
the expression \eqref{;arewr} is equal to 
\begin{align*}
 \int _{0}^{\infty }\frac{dv}{v}\,\exp \left( 
 -\frac{1}{2v}
 \right) \exp \left\{ 
 -\frac{(r+\la )^{2}}{2}v
 \right\} 
 \ex \!\left[ 
 \cos \bigl\{ 
 \sqrt{\xi ^{2}-\la ^{2}}\beta (v)
 \bigr\} 
 \right] . 
\end{align*}
Since there holds the equality 
\begin{align*}
 \ex \!\left[ 
 \cos \bigl\{ 
 \sqrt{\xi ^{2}-\la ^{2}}\beta (v)
 \bigr\} 
 \right] 
 =\exp \left( 
 -\frac{\xi ^{2}-\la ^{2}}{2}v
 \right) , 
\end{align*}
the above integral with respect to $v$ is equal to 
\begin{align*}
 2K_{0}\bigl( \sqrt{(r +\la )^{2}+\xi ^{2}-\la ^{2}}\,\bigr) 
\end{align*}
by \eqref{;k0}, which proves that \eqref{;arewr} agrees with 
\eqref{;jltreprd}. 

By replacing 
$\cos \bigl( \sqrt{\xi ^{2}-\la ^{2}}\sinh x\bigr) $ 
in \eqref{;arewr} by 
$\cosh \bigl( \sqrt{\la ^{2}-\xi ^{2}}\sinh x\bigr) $, 
the case $\la \ge |\xi |$ is treated in a similar way owing to the fact that 
\begin{align*}
 \ex \!\left[ 
 \cosh \bigl\{ \sqrt{\la ^{2}-\xi ^{2}}\beta (v)\bigr\} 
 \right] 
 &=\ex \!\left[ 
 e^{\sqrt{\la ^{2}-\xi ^{2}}\beta (v)}
 \right] \\
 &=\exp \left( 
 \frac{\la ^{2}-\xi ^{2}}{2}v
 \right) , 
\end{align*}
where the first equality is due to symmetry of Brownian motion. 

\subsection{Explanation of \eqref{;hw} via \eqref{;ltr}}\label{;dhw}

In \cite{yor80} (see also \cite[Appendix~A]{mySI}), 
Yor obtained the integral representation \eqref{;hw} 
for $\Theta _{r}(t)$ by inverting its Laplace 
transform \eqref{;ltt} taken with respect to variable 
$t$. In the last part of this appendix, we explain 
\eqref{;hw} via \eqref{;ltr}, the Laplace transform 
with respect to variable $r$. 
 
We fix $t>0$ below. By the representation \eqref{;exphw} 
and Fubini's theorem, the left-hand side of the relation 
\eqref{;ltr} may be written as 
\begin{align*}
 \frac{1}{2\pi }\exp \left( 
 \frac{\pi ^{2}}{2t}
 \right) 
 \ex \!\left[ 
 \frac{
 \sinh B_{t}\sin \left( \pi B_{t}/t\right) 
 }
 {
 \cosh B_{t}+\cosh x
 }
 \right] ,\quad x\in \R . 
\end{align*}
Hence if we take the Fourier transform on both sides 
of \eqref{;ltr}, it reads 
\begin{align*}
 \frac{1}{2\pi }\exp \left( 
 \frac{\pi ^{2}}{2t}
 \right) 
 \int _{\R }dx\,\cos (\xi x)\ex \!\left[ 
 \frac{
 \sinh B_{t}\sin \left( \pi B_{t}/t\right) 
 }
 {
 \cosh B_{t}+\cosh x
 }
 \right] 
 =\exp \left( 
 -\frac{\xi ^{2}}{2}t
 \right) ,\quad \xi \in \R .
\end{align*}
Therefore thanks to the injectivity of Fourier and Laplace 
transforms, in order to verify the representation \eqref{;hw}, 
it suffices to show that for any $\xi \in \R $, 
\begin{align}\label{;aim}
 \int _{\R }dx\,\cos (\xi x)\ex \!\left[ 
 \frac{
 \sinh B_{t}\sin \left( \pi B_{t}/t\right) 
 }
 {
 \cosh B_{t}+\cosh x
 }
 \right] 
 =2\pi \exp \left( 
 -\frac{\pi ^{2}}{2t}-\frac{\xi ^{2}}{2}t
 \right) .
\end{align}

Let $\xi \neq 0$ for a while and note the fact that 
for any $b\in \R $ with $b\neq 0$, 
\begin{align}\label{;intrel}
 \int _{\R }dx\,
 \frac{\cos (\xi x)}
 {\cosh b+\cosh x}
 =\frac{2\pi \sin (\xi b)}
 {\sinh (\pi \xi )\sinh b}. 
\end{align}
Indeed, it is known (cf. \cite[Subsection~13.21, Equation~(9)]{wat}) 
that 
\begin{align}\label{;phw}
 \int _{0}^{\infty }du\,e^{-u\cosh b}K_{i\xi }(u)
 =\frac{\pi \sin (\xi b)}
 {\sinh (\pi \xi )\sinh b}, 
\end{align}
into the left-hand side of which we plug the integral 
representation \eqref{;intk1i} of $K_{i\xi }$ 
to obtain by Fubini's theorem, 
\begin{align*}
  \int _{0}^{\infty }du\,e^{-u\cosh b}K_{i\xi }(u)
  &=\frac{1}{2}\int _{\R }dx\,\cos (\xi x)
  \int _{0}^{\infty }du\,e^{-u(\cosh b+\cosh x)}\\
  &=\frac{1}{2}\int _{\R }dx\,
 \frac{\cos (\xi x)}
 {\cosh b+\cosh x}. 
\end{align*}
Using the above fact \eqref{;intrel} and Fubini's 
theorem, we then see that the left-hand side of 
\eqref{;aim} is equal to 
\begin{align*}
 &\frac{2\pi }{\sinh (\pi \xi )}
 \ex \!\left[ 
 \sinh B_{t}\sin \left( \frac{\pi }{t}B_{t}\right)
 \frac{\sin (\xi B_{t})}{\sinh B_{t}};\,B_{t}\neq 0 
 \right] \\
 &=\frac{\pi }{\sinh (\pi \xi )}
 \ex \!\left[ 
 \cos \left\{ 
 \left( 
 \frac{\pi }{t}-\xi 
 \right) B_{t}
 \right\} 
 -\cos \left\{ 
 \left( 
 \frac{\pi }{t}+\xi 
 \right) B_{t}
 \right\} 
 \right] \\
 &=\frac{\pi }{\sinh (\pi \xi )}
 \left[ 
 \exp \left\{ 
 -\frac{1}{2}\left( 
 \frac{\pi }{t}-\xi 
 \right) ^{2}t
 \right\} 
 -\exp \left\{ 
 -\frac{1}{2}\left( 
 \frac{\pi }{t}+\xi 
 \right) ^{2}t
 \right\} 
 \right] , 
\end{align*}
which agrees with the right-hand side of \eqref{;aim}. 
The validity of \eqref{;aim} in the case $\xi =0$ 
is now clear since both sides of \eqref{;aim} 
are continuous functions in $\xi $. 

If we consider the meromorphic function 
\begin{align*}
 f(z)=\frac{\cos (\xi z)}{\cosh b+\cosh z}
\end{align*} 
on $\C $, then by noting the fact that 
the poles $w$ of $f$ each of whose 
imaginary part $\mathrm{Im}\,w$ satisfies 
$0<\mathrm{Im}\,w<2\pi $ are $\pm b+\pi i$, 
the above formula \eqref{;intrel} may be verified 
by standard residue calculus along a rectangular 
contour circling $\pm b+\pi i$ and having 
its two sides on the two lines 
$\mathrm{Im}\,z=0$ and $\mathrm{Im}\,z=2\pi .$ 
We also note that the formula \eqref{;phw} may be seen from the 
following particular case of the Hankel--Lipschitz formulae 
(cf. \cite[p.~703]{gr}, \cite[Subsection~13.21]{wat}): 
for every $\nu \in \C $ whose real part is strictly greater than $-1$, 
\begin{align*}
 \int _{0}^{\infty }du\,e^{-u\cosh b}I_{\nu }(u)
 =\frac{e^{-\nu b}}{\sinh b},\quad b>0, 
\end{align*}
together with the definition of $K_{\nu }$ for a noninteger $\nu $ 
(see \cite[Section~5.7]{leb}): 
\begin{align*}
 K_{\nu }(b)=\frac{\pi }{2}\frac{I_{-\nu }(b)-I_{\nu }(b)}{\sin (\pi \nu )}. 
\end{align*}

\begin{rem}\label{;rfinal}
The same computation as in the above verification 
of the formula \eqref{;aim} also proves that when 
$t>0$ and $\xi \neq 0$, 
\begin{align}\label{;aimd}
 \int _{\R }dx\,\cos (\xi x)\ex \!\left[ 
 \frac{
 \sinh B_{t}\sin \left( \al B_{t}/t\right) 
 }
 {
 \cosh B_{t}+\cosh x
 }
 \right] 
 =2\pi \exp \left( 
 -\frac{\al ^{2}}{2t}-\frac{\xi ^{2}}{2}t
 \right) 
 \frac{\sinh (\al \xi )}{\sinh (\pi \xi )} 
\end{align}
for any $\al \in \R $. Fourier inversion of the right-hand side 
gives us several formulae for expectations as in the 
left-hand side of \eqref{;aimd} for different values of $\al $. 
When $\al =2\pi $ for example, since it is readily seen that 
for any $\xi \in \R $, 
\begin{align*}
 \int _{\R }dx\,\cos (\xi x)\frac{1}{\sqrt{2\pi t}}
 \exp \left( 
 -\frac{x^{2}}{2t}
 \right) \cos \left( 
 \frac{\pi }{t}x
 \right) 
 =\exp \left( 
 -\frac{\pi ^{2}}{2t}-\frac{\xi ^{2}}{2}t
 \right) \cosh (\pi \xi ), 
\end{align*}
we have 
\begin{align*}
 \ex \!\left[ 
 \frac{
 \sinh B_{t}\sin \left( 2\pi B_{t}/t\right) 
 }
 {
 \cosh B_{t}+\cosh x
 }
 \right] 
 =2\sqrt{\frac{2\pi }{t}}
 \exp \left( 
 -\frac{3\pi ^{2}}{2t}-\frac{x^{2}}{2t}
 \right) \cos \left( 
 \frac{\pi }{t}x
 \right) , \quad x\in \R , 
\end{align*}
which, after replacing $t$ and $x$ by 
$4t$ and $2x$, respectively, agrees with \cite[Lemma~3.1]{my2003}. 
In the case $\al =\pi /2$, the right-hand side of \eqref{;aimd} 
becomes 
\begin{align*}
 \pi \exp \left( 
 -\frac{\pi ^{2}}{8t}-\frac{\xi ^{2}}{2}t
 \right) \frac{1}{\cosh (\pi \xi /2)}, 
\end{align*}
and noting the fact that $1/\cosh (\pi \xi /2),\,\xi \in \R $, is 
the characteristic function of 
$\log |C|$ whose probability density is given by 
$1/(\pi \cosh x),\,x\in \R $ (cf. \cite[Chapter~0, Section~6]{ry}), 
we have the relation 
\begin{align*}
 \ex \!\left[ 
 \frac{
 \sinh B_{t}\sin \left( \frac{\pi }{2t}B_{t}\right) 
 }
 {
 \cosh B_{t}+\cosh x
 }
 \right] 
 =\exp \left( 
 -\frac{\pi ^{2}}{8t}
 \right) 
 \ex \!\left[ 
 \frac{1}{\cosh (x+B_{t})}
 \right] 
\end{align*}
for any $x\in \R $. By rewriting 
\begin{align*}
 \frac{1}{\cosh (x+B_{t})}
 =\frac{2(\cosh x\cosh B_{t}-\sinh x\sinh B_{t})}
 {\cosh (2B_{t})+\cosh (2x)}, 
\end{align*}
it also holds that 
\begin{align*}
 \ex \!\left[ 
 \frac{
 \sinh B_{t}\sin \left( \frac{\pi }{2t}B_{t}\right) 
 }
 {
 \cosh B_{t}+\cosh x
 }
 \right] 
 =2\exp \left( 
 -\frac{\pi ^{2}}{8t}
 \right) \cosh x\,
 \ex \!\left[ 
 \frac{\cosh B_{t}}{\cosh (2B_{t})+\cosh (2x)}
 \right] . 
\end{align*}
Finally, by dividing both sides of \eqref{;aimd} by $\al \neq 0$ 
and letting $\al \to 0$, it follows that 
\begin{align*}
 \frac{1}{t}\ex \!\left[ 
 \frac{
 B_{t}\sinh B_{t}
 }
 {
 \cosh B_{t}+\cosh x
 }
 \right] 
 =\ex \!\left[ 
 \frac{1}{\cosh (x+B_{t})+1}
 \right] ,\quad x\in \R , 
\end{align*}
because of the fact that $\pi \xi /\sinh (\pi \xi ),\,\xi \in \R $, is 
the Fourier transform of the probability density 
$
\bigl\{ 2(\cosh x+1)\bigr\} ^{-1},\,x\in \R 
$, which may be seen from \eqref{;intrel} by letting $b\to 0$ 
on both sides (see also \cite{ry} as cited in the case 
$\al =\pi /2$ above). 
\end{rem}
\bigskip 

\noindent 
{\bf Acknowledgements.} 
This work was partially supported by JSPS KAKENHI Grant Number 17K05288. 


\end{document}